\documentclass[final,1p,times,margin=1in]{elsarticle}

\usepackage{amssymb}
 \usepackage{amsthm}
\usepackage{amscd}
\usepackage{amsmath}
\usepackage{amsfonts}
\usepackage{color}
\usepackage{xcolor}
\usepackage{colortbl}
\usepackage{graphicx}
\usepackage{tikz}
\usepackage{url}
\newtheorem{theorem}{Theorem}[section]

\newtheorem{example}[subsection]{Example}
\usepackage[ruled]{algorithm2e}
\usepackage{mathrsfs}
\usepackage{titletoc}
\usepackage{float}
\usepackage{graphicx}
\usepackage{subcaption}

\newcommand{\first}[1]{%
    \cellcolor{red!40}\textbf{#1}%
}
\newcommand{\second}[1]{%
    \cellcolor{orange!40}\textbf{#1}%
}

\usepackage{caption}
\usepackage{subcaption}
\usepackage{booktabs}
\usepackage{geometry}

\journal{******}

\begin{document}

\begin{frontmatter}

\title{Evolving edge weights via local entropy flow and cohesion flow on graphs}

\author[ruc]{Juan Zhao}
\ead{zhaojuan0509@ruc.edu.cn}

\author[ruc]{Jicheng Ma}
\ead{2019202433@ruc.edu.cn}

\author[ruc]{Yunyan Yang\corref{cor1}}
\ead{yunyanyang@ruc.edu.cn}

\author[bnu]{Liang Zhao}
\ead{liangzhao@bnu.edu.cn}

\cortext[cor1]{Corresponding author}

\address[ruc]{School of Mathematics, Renmin University of China, Beijing, 100872, China}
\address[bnu]{School of Mathematical Sciences, Key Laboratory of Mathematics and Complex Systems of MOE,\\
Beijing Normal University, Beijing, 100875, China}

\begin{abstract}

In this paper, we first propose two different quantities on graphs, namely local entropy and cohesion, then design two corresponding flows for edge weights: the local entropy flow and the cohesion flow.
We establish the global existence and uniqueness of solutions for both flows and investigate their asymptotic behaviors, including the case that the limit goes to positive infinity. 
Moreover, they can be applied to fundamental network analysis tasks, including community detection and node classification.
Empirical evaluations demonstrate that our method achieves performance competitive with  Ollivier Ricci flow and Lin-Lu-Yau Ricci flow on benchmark network analysis tasks. In experimental scenarios, we first apply the cohesion flow to evolve the edge weights of the graph, and then apply the local entropy flow to further update the resulting weighted graph.
Both flows are computationally efficient, leading to a significant reduction in overall computational cost and improved scalability.
\end{abstract}


\end{frontmatter}

\titlecontents{section}[0mm]
                       {\vspace{.2\baselineskip}}
                       {\thecontentslabel~\hspace{.5em}}
                        {}
                        {\dotfill\contentspage[{\makebox[0pt][r]{\thecontentspage}}]}
\titlecontents{subsection}[3mm]
                       {\vspace{.2\baselineskip}}
                       {\thecontentslabel~\hspace{.5em}}
                        {}
                       {\dotfill\contentspage[{\makebox[0pt][r]{\thecontentspage}}]}

\setcounter{tocdepth}{2}



\numberwithin{equation}{section}
\section{Introduction}

Graph models serve as a fundamental tool across a wide range of fields, including network analysis, machine learning, and complex systems. In these models, edges between nodes are assigned weights to encode the strength or distance of pairwise connections. Understanding the intrinsic structure of graphs, such as node clustering into communities, or information propagation patterns across the network, is essential for tasks including community detection, core detection, and graph representation learning.

An important and widely adopted class of methods for characterizing the intrinsic topological and geometric properties of discrete graph structures is discrete curvature and curvature flow. Ollivier proposed a discrete Ricci curvature framework for metric measure spaces based on optimal transport distance in the seminal works \cite{Ollivier-1, Ollivier-2}. This framework was later modified by Lin, Lu and Yau \cite{Lin1}, and is widely referred to as Lin-Lu-Yau Ricci curvature. Bauer, Jost and Liu systematically investigated the spectral properties of the normalized graph Laplacian associated with Ollivier Ricci curvature \cite{Bauer-Jost-Liu}. Jost and Liu established curvature-dimension inequalities for discrete graph settings \cite{Jost-Liu}.

Inspired by Hamilton's pioneering Ricci flow on smooth manifolds \cite{Hamilton} and Perelman's manifold surgery theory \cite{perelman2002entropy}, discrete geometric flows have aspired to generalize continuous manifold geometry to discrete graph structures. Chow and Luo developed the theory of combinatorial Ricci flow on surfaces \cite{CL}. For complex network analysis, Weber, Saucan and Jost introduced Forman-Ricci flow and systematically explored its applications in graph structural mining \cite{Weber1,Weber2}. Focusing on community detection, Ni et al. applied Ollivier Ricci flow with graph surgery operations to identify community structures \cite{NLLG}. This work was further generalized to normalized Lin-Lu-Yau Ricci flow by Lai, Bai and Lin \cite{LaiX}. On the theoretical side, Bai et al. established the well-posedness theory for Lin-Lu-Yau Ricci flow, including the existence, uniqueness, and long-time convergence of its solutions \cite{Bai-Lin}. For recent theoretical advances on the well-posedness, convergence behavior, and prescribed curvature problem of Ollivier or Lin-Lu-Yau Ricci curvature, we refer interested readers to \cite{bai,Bai-Hua-Lin-Liu,Bai-Li-Liu-Lai,Bai-Liu-Lai,Li-Munch,Lin-Liu}. In addition, Ma and Yang proposed several variants of the normalized Ricci flow, established their long-time existence theory, and further applied these flow frameworks to community detection tasks \cite{Ma1,Ma2,Ma3}. For extended applications of discrete curvature flow in core detection and hypergraph community detection, we refer readers to \cite{Sengupta,Tian-Ma-Yang-Zhao,Zhao-Ma-Yang-Zhao2,Zhao-Ma-Yang-Zhao}.

Entropy is a foundational concept in both physics and information theory. Entropy-based approaches have long been proven effective for quantifying uncertainty, complexity, and information dynamics in networked systems. Rooted in information theory and statistical mechanics, entropy and relative entropy have been extensively studied in connection with stochastic processes, diffusion dynamics, and random walks on graphs \cite{Shannon,Kullback,Cover}. These methods are widely used in network analysis to measure key structural properties including heterogeneity, robustness, and node centrality \cite{Anand, Dehmer}.

In the context of graph frameworks, the key information metric we adopt is the Kullback-Leibler (KL) divergence, which quantifies the dissimilarity between two probability measures. In our recent work \cite{entropyflow}, we constructed a family of nowhere-zero random walks on graphs, and defined edge entropy via a symmetric variant of KL divergence (J-divergence \cite{Jeff}) between these walks. Motivated by the geometric intuition of Ricci flow, we further proposed an entropy flow on weighted graphs with rigorous theoretical guarantees. This method achieves detection accuracy comparable to discrete Ollivier Ricci flow in the task of community detection, while avoiding the high computational cost of optimal transport. Despite its improved efficiency over Ricci flow methods, this framework still faces scalability challenges for large-scale graph computations, due to the global nature of the nowhere-zero random walks it relies on.

In this paper, to further simplify the computation of global entropy, we introduce a local entropy and the corresponding local entropy flow. Along this flow, edges with higher entropy experience faster weight growth, while those with lower entropy grow more slowly. In addition to the point of view of probability, we introduce a cohesion, characterizing  the geometric structure through cohesiveness of the two endpoints of an edge. We further construct the cohesion flow to evolve edge weights accordingly. Theoretically, we prove the existence and uniqueness of global solutions to both flows, and further investigate their asymptotic behavior.
To assess the effectiveness of the proposed flows, we consider two representative network analysis tasks: community detection and node classification using graph convolutional networks (GCNs).
For community detection, we observe that both the local entropy flow and the cohesion flow are effective when applied individually, while their sequential combination, which first evolves edge weights via the cohesion flow and then further updates them using the local entropy flow, yields improved performance.
For node classification, both the individual flows and their combination achieve comparable results.
Experimental results show that the proposed  flows achieve performance compared with both Ricci flow approaches \cite{LaiX,Ma1,Ma2,Ma3,NLLG,GEGCN} and our previous global entropy flow approaches \cite{entropyflow}.

The remainder of this paper is organized as follows. 
In Section \ref{mainresults}, we introduce the flows and state the main results.
In Section \ref{examples}, we provide several examples of both flows. Section \ref{theorems} presents the proofs of the  main theorems. In Section \ref{experimental}, we validate the performance of the proposed flows for community detection and node classification. Finally, we conclude this work in Section \ref{conclusion}.

\section{Local entropy flow and cohesion flow}
\label{mainresults}

In this section, we present two quantities and their associated flows on graphs. We first introduce the local entropy defined via paired local random walks and construct the corresponding local entropy flow, together with its well-posedness and convergence properties. We then define the cohesion based on local neighborhood structures and construct the corresponding cohesion flow, along with its well-posedness and long-time behavior.

\subsection{Local entropy flow}
In this subsection, we define local random walks and the associated edge entropy.
Based on these quantities, we construct the local entropy flow and present its main theoretical properties together with a discrete version for numerical implementation.

\subsubsection{Local random walks}

Let $G=(V, E, \mathbf{w})$ be a connected finite weighted graph, where $V=\{x_1, x_2, \dots, x_n\}$ is the vertex set, $E=\{e_1,e_2,\dots, e_m\}$ is the edge set, and $\mathbf{w}=(w_e)_{e\in E}$ is the edge weight vector with $w_e>0$ for all $e\in E$. We write $x\sim y$ if vertices $x$ and $y$ are adjacent. For any edge $e=xy\in E$, we denote the $1$-step neighborhood of $x$ and $y$ by
$$\mathscr{N}_x = \{u\in V: u\sim x \text{ or } u=x\}, \quad \mathscr{N}_y = \{u\in V: u\sim y \text{ or } u=y\},$$
respectively.

For each $e=xy$ and a fixed parameter $\alpha\in(0,1)$, the {\it local random walk starting from} $x$ (associated with edge $e=xy$), denoted by $\mu_x^\alpha: V\to [0,1]$, is defined as follows:
\begin{itemize}
    \item If $\mathscr{N}_y\setminus\mathscr{N}_x \neq \emptyset$, we set
    \begin{equation}\label{mu-1}
    \mu_x^\alpha(z)=
    \begin{cases}
    \alpha,& z=x,\\[1.5ex]
    \alpha(1-\alpha)\frac{w_{xz}}{\sum\limits_{u\sim x}w_{xu}},& z\sim x,\\[1.2ex]
    (1-\alpha)^2\frac{w_{yz}}{\sum\limits_{u\in\mathscr{N}_y\setminus\mathscr{N}_x}w_{yu}},& z\in\mathscr{N}_y\setminus\mathscr{N}_x,\\[1.2ex]
    0,& \text{otherwise}.
    \end{cases}
    \end{equation}
    \item If $\mathscr{N}_y\setminus\mathscr{N}_x = \emptyset$, we define $\mu_x^\alpha$ as the $1$-step $\alpha$-lazy random walk on $G$, i.e.,
    \begin{equation}\label{mu-2}
    \mu_x^\alpha(z)=
    \begin{cases}
    \alpha,& z=x,\\[1.5ex]
    (1-\alpha)\frac{w_{xz}}{\sum\limits_{u\sim x}w_{xu}},& z\sim x,\\[1.2ex]
    0,& \text{otherwise}.
    \end{cases}
    \end{equation}
\end{itemize}

It is straightforward to verify that for any $\alpha\in(0,1)$, the support of $\mu_x^\alpha$ is exactly $\mathscr{N}_x\cup\mathscr{N}_y$, i.e., $\{u\in V: \mu_x^\alpha(u)\neq 0\} = \mathscr{N}_x\cup\mathscr{N}_y$. In particular, we have
\begin{equation}\label{positive}
0<\mu_x^\alpha(z)<1,\quad\forall z\in \mathscr{N}_x\cup\mathscr{N}_y
\end{equation}
and
\begin{equation}\label{identity}
\sum_{z\in V}\mu_x^\alpha(z) = \sum_{z\in \mathscr{N}_x\cup\mathscr{N}_y}\mu_x^\alpha(z)=1,
\end{equation}
which confirms that $\mu_x^\alpha$ is a well-defined probability measure on $V$.

\subsubsection{Entropy between local random walks}

For any $\alpha\in(0,1)$ and edge $e=xy\in E$, let $\mu_x^\alpha$ and $\mu_y^\alpha$ be paired local random walks associated with $e$ defined above. We define the edge entropy of $e$ as the symmetric KL divergence between $\mu_x^\alpha$ and $\mu_y^\alpha$:
\begin{equation}\label{theta}
\Theta_e^\alpha=\sum_{z\in \mathscr{N}_x\cup\mathscr{N}_y}\left(\mu_x^\alpha(z)\log\frac{\mu_x^\alpha(z)}{\mu_y^\alpha(z)}
+\mu_y^\alpha(z)\log\frac{\mu_y^\alpha(z)}{\mu_x^\alpha(z)}\right).
\end{equation}

Since the function $f(s)=-\log s$ is convex on $(0,+\infty)$, we apply Jensen's inequality to \eqref{positive} and \eqref{identity} to derive the non-negativity of the edge entropy:
\begin{align}\nonumber
\Theta_e^\alpha&=-\sum_{z\in \mathscr{N}_x\cup\mathscr{N}_y}\mu_x^\alpha(z)\log\frac{\mu_y^\alpha(z)}{\mu_x^\alpha(z)}
-\sum_{z\in \mathscr{N}_x\cup\mathscr{N}_y}\mu_y^\alpha(z)\log\frac{\mu_x^\alpha(z)}{\mu_y^\alpha(z)}\\ \nonumber
&\geq-\log\left(\sum_{z\in \mathscr{N}_x\cup\mathscr{N}_y}\mu_x^\alpha(z)\frac{\mu_y^\alpha(z)}{\mu_x^\alpha(z)}\right)
-\log\left(\sum_{z\in \mathscr{N}_x\cup\mathscr{N}_y}\mu_y^\alpha(z)\frac{\mu_x^\alpha(z)}{\mu_y^\alpha(z)}\right)\\
&=0.\label{greater}
\end{align}
Moreover, $\Theta_e^\alpha=0$ if and only if $\mu_x^\alpha(z)=\mu_y^\alpha(z)$ for all $z\in \mathscr{N}_x\cup\mathscr{N}_y$.

\subsubsection{Local entropy flow}

For any $\alpha\in(0,1)$ and edge $e=xy\in E$, let $w_{0,e}>0$ be the initial weight of $e$. We define the local random walk-based entropy flow as:
\begin{equation}\label{entropy-flow}
\begin{cases}
w_e^\prime(t)=\Theta_e^\alpha(t),& t>0,\\[1.5ex]
w_e(t)>0, & t>0,\\[1.5ex]
w_e(0)=w_{0,e},& \forall e\in E.
\end{cases}
\end{equation}
Here $\Theta_e^\alpha(t)$ is the edge entropy of $e$ computed with respect to the time-varying weight vector $\mathbf{w}(t)=(w_e(t))_{e\in E}$ at time $t$, according to \eqref{theta}. Since $\Theta_e^\alpha(t)\geq 0$ for all $t\geq0$ by \eqref{greater}, the derivative $w_e'(t)$ is non-negative along the flow. It follows that $w_e(t)\geq w_e(0)=w_{0,e}>0$ for all $t\geq0$ and every $e\in E$.

\subsubsection{Main results}

Let $\mathbb{R}_+^m = \{\mathbf{a}=(a_1,a_2,\dots,a_m)\in\mathbb{R}^m: a_j>0 \text{ for all } 1\leq j\leq m\}$. The following theorem establishes the existence, uniqueness, and long-time convergence of global solutions to the local entropy flow (\ref{entropy-flow}).

\begin{theorem}\label{existence}
Let $G=(V,E,\mathbf{w}_0)$ be a connected finite weighted graph with initial weight $\mathbf{w}_0\in\mathbb{R}_+^m$. Then for any $\alpha\in(0,1)$, the local entropy flow \eqref{entropy-flow} admits a unique global solution $\mathbf{w}(t)=(w_e(t))_{e\in E}\in\mathbb{R}_+^m$ for all $t\in[0,+\infty)$.

Moreover, exactly one of the following two alternatives holds for the solution:
\begin{enumerate}
    \item[(i)] For every $e\in E$, $w_e(t)$ is non-decreasing and $\lim\limits_{t\to+\infty} w_e(t) = +\infty$;
    \item[(ii)] For every $e\in E$, $w_e(t)$ is non-decreasing and converges to some positive limit $w_e^\ast>0$ as $t\to+\infty$. Moreover, $\lim\limits_{t\to+\infty} \Theta_e^\alpha(t) = 0$ for all $e\in E$, where $\Theta_{e}^\ast$ is the entropy on the edge $e$ with respect to the weights $\mathbf{w}^\ast=(w_e^\ast)_{e\in E}$.
\end{enumerate}
\end{theorem}

If $G$ is disconnected, the above theorem holds independently on each connected component, since the local entropy flow does not interact across different components.

Next, we introduce a discrete version of the continuous local entropy flow \eqref{entropy-flow} for numerical implementation. Fix a time step size $s>0$, and let $t_j = js$ for $j=0,1,2,\cdots$. The discrete local entropy flow is defined as:
\begin{equation}\label{discrete}
\begin{cases}
w_e(t_{j+1})=w_e(t_j)+s\Theta_e^\alpha(t_j),\\[1.5ex]
w_e(t_j)>0,\quad j=0,1,2,\cdots,\\[1.5ex]
w_e(0)=w_{0,e},\quad \forall e\in E.
\end{cases}
\end{equation}
Concerning the existence, uniqueness, and convergence of a solution to  \eqref{discrete}, we have an analogous result to Theorem 2.1 in \cite{entropyflow}.

\begin{theorem}\label{discrete-existence}
Let $G=(V,E,\mathbf{w}_0)$ be a connected finite weighted graph as in Theorem \ref{existence}. Then for any $\alpha\in(0,1)$ and initial weight $\mathbf{w}_0\in\mathbb{R}_+^m$, the discrete local entropy flow \eqref{discrete} admits a unique solution $\mathbf{w}(t_j)=(w_e(t_j))_{e\in E}$ for all $j=0,1,2,\cdots$.

Moreover, exactly one of the following two alternatives holds:
\begin{enumerate}
    \item[(i)] For every $e\in E$, $w_e(t_j)$ is non-decreasing and $\lim\limits_{j\to+\infty} w_e(t_j) = +\infty$;
    \item[(ii)] For every $e\in E$, $w_e(t_j)$ is non-decreasing and converges to some positive limit $w_e^\ast>0$ as $j\to+\infty$. In this case, $\lim\limits_{j\to+\infty}\Theta_e^\alpha(t_j)=0$ for all $e\in E$.
\end{enumerate}
\end{theorem}

We note that the support of the local random walk at the edge $e=xy$ is $B_1(x)\cup B_1(y)$, where $B_1(x)$ denotes the ball centered at $x$ with radius $1$. Here, the radius denotes the combinatorial distance. Similarly, for any integer $\ell\geq 1$, if $B_\ell(x)\cup B_\ell(y)$ is a non-trivial subset of $V$, one can define local random walks supported on $B_\ell(x)\cup B_\ell(y)$. The nowhere-zero random walk in \cite{entropyflow} corresponds to the extreme case where the support is the entire set of vertex $V$.
By replacing $\mu_x^\alpha$ and $\mu_y^\alpha$ with random walks supported in $B_\ell(x)\cup B_\ell(y)$, we can define the corresponding edge entropy $\Theta_e^\alpha$ and entropy flow. The analogous results of Theorems \ref{existence} and \ref{discrete-existence} still hold for such generalized flows.

\subsection{Cohesion flow}
In this subsection, we define the cohesion  based on local neighborhood structures and construct the corresponding cohesion flow. We then present the main theoretical results on its well-posedness and long-time behavior, together with a discrete version for numerical implementation.

\subsubsection{Cohesion}
Let $G=(V,E,\mathbf{w})$ be a connected finite weighted graph, where
$V=\{x_1,x_2,\dots,x_n\}$
is the vertex set,
$E=\{e_1,e_2,\dots,e_m\}$
is the edge set, and
$\mathbf{w}=(w_e)_{e\in E}$
is the edge weight vector satisfying $w_e>0$ for all $e\in E$.
For any vertex $x\in V$, we denote its closed neighborhood by
$$\mathscr{N}_x = \{u\in V: u\sim x \text{ or } u=x\}, \quad \mathscr{N}_y = \{u\in V: u\sim y \text{ or } u=y\}.$$
For any edge $e=xy\in E$, we denote the set of common neighbors of $x$ and $y$ by
$$
\mathscr N_{xy}
=
\{z\in V:z\sim x \text{ and } z\sim y\}.
$$
To characterize the local structural relationship around the edge $e=xy$, we define
\begin{equation}\label{cohesion-A}
A_{xy}
=
\sum_{u\in \mathscr N_x\setminus \mathscr N_y}\frac1{w_{xu}}
+
\sum_{u\in \mathscr N_y\setminus \mathscr N_x}\frac1{w_{yu}},
\end{equation}
and
\begin{equation}\label{cohesion-B}
B_{xy}
=
\sum_{z\in \mathscr N_{xy}}
\left(
\frac1{w_{xz}}
+
\frac1{w_{yz}}
\right).
\end{equation}
We then define the quantity of local structure associated with the edge $e=xy$ by
\begin{equation}\label{cohesion-tilde}
\widetilde{C}_{xy} =
\begin{cases}
\displaystyle \frac{A_{xy}-B_{xy}}{A_{xy}+B_{xy}}, & \text{if } A_{xy}+B_{xy}>0,\\[1.5em]
0, & \text{if } A_{xy}+B_{xy}=0,
\end{cases}
\end{equation}
where $A_{xy}$ and $B_{xy}$ are given by \eqref{cohesion-A} and \eqref{cohesion-B}, respectively.
The cohesion of $e=xy$ is defined by
\begin{equation}\label{cohesion-C}
C_{xy} = \exp(\widetilde{C}_{xy}),
\end{equation}
where $\widetilde{C}_{xy}$ is defined in \eqref{cohesion-tilde}.
Since $-1\leq \widetilde C_{xy}\leq 1$, it follows from \eqref{cohesion-C} that $e^{-1}\leq C_{xy}\leq e$.

\subsubsection{Cohesion flow}
For any edge $e=xy\in E$, let $w_{0,e}>0$ be the initial weight of e. We define the cohesion flow as
\begin{equation}\label{cohesion-flow}
\begin{cases}
w_e^{\prime}(t) = C_e(t), & t>0,\\[1.2ex]
w_e(t) > 0, & t>0,\\[1.2ex]
w_e(0) = w_{0,e}, & \forall e\in E.
\end{cases}
\end{equation}
Here $C_e(t)$ denotes the cohesion of edge $e$ computed with respect to weight vector $\mathbf{w}(t)=(w_e(t))_{e \in E}$ at time $t$.
Since $C_e(t)>0$ for all $t\geq 0$, we have $w_e^{\prime}(t)>0$, and hence $w_e(t)$ is strictly increasing and satisfies $w_e(t)\geq w_{e}(0)=w_{0,e}>0$ for all $t\geq 0$ and every $e \in E$.

Next, we introduce a discrete version of the continuous cohesion flow \eqref{cohesion-flow} for numerical implementation. Fix a time step size $s>0$, and let $t_j=js$, $j=0,1,2,\cdots$. The discrete cohesion flow is defined as
\begin{equation}\label{discrete-cohesion-flow}
\begin{cases}
w_e(t_{j+1}) = w_e(t_j)+sC_e(t_j),\\[1.5ex]
w_e(t_j)>0,\quad j=0,1,2,\cdots,\\[1.5ex]
w_e(0)=w_{0,e},\quad \forall e\in E.
\end{cases}
\end{equation}
Here $C_e(t_j)$ denotes the cohesion computed with respect to the weight vector $\mathbf{w}(t_j)=(w_e(t_j))_{e\in E}$. Since $C_e(t_j)>0$ for all $j\geq0$, the sequence $\{w_e(t_j)\}_{j=0}^{\infty}$ is strictly increasing for every edge $e\in E$.

\subsubsection{Main results}
Analogous to the results of the local entropy flow, we have the following theorems for the cohesion flow.

\begin{theorem}\label{cohesion-existence}
Let $G=(V,E,\mathbf{w}_0)$ be a connected finite weighted graph with initial weight $\mathbf{w}_0\in\mathbb{R}_+^m$. Then the cohesion flow \eqref{cohesion-flow} admits a unique global solution $\mathbf{w}(t)=(w_e(t))_{e\in E}\in\mathbb{R}_+^m$ for all $t\in[0,+\infty)$. Moreover, for every edge $e\in E$, $w_e(t)$ is strictly increasing and $\lim\limits_{t\to+\infty} w_e(t) = +\infty$.
\end{theorem}

Next, we establish the corresponding result for the discrete cohesion flow \eqref{discrete-cohesion-flow}.

\begin{theorem}\label{discrete-cohesion-existence}
Let $G=(V,E,\mathbf{w}_0)$ be a connected finite weighted graph with $\mathbf{w}_0\in\mathbb{R}_+^m$. Then for any step size $s>0$, the discrete cohesion flow \eqref{discrete-cohesion-flow} admits a unique solution $\mathbf{w}(t_j)=(w_e(t_j))_{e\in E}$ for all $j=0,1,2,\cdots$. Moreover, for every edge $e\in E$,  $w_e(t_j)$ is strictly increasing and $\lim\limits_{j\to+\infty} w_e(t_j) = +\infty$.
\end{theorem}

The proofs of Theorems \ref{existence}- \ref{discrete-cohesion-existence} will be presented in Section \ref{theorems}.

\section{Examples}\label{examples}
In this section, we investigate the behavior of the local entropy flow and the cohesion flow through several examples. We first present the computation of the two flows in both continuous and discrete settings, together with the corresponding edge weight evolution processes. We then show that both flows can reveal the community structure of a network. Finally, we provide a representative example to demonstrate that the two flows may exhibit different behaviors in certain cases.

We begin with an example of the continuous flows.

\begin{example}\label{ex-01}
Let $\alpha\in(0,1)$, and let $G=(V,E,\mathbf{w}_0)$ be the regular hexagon shown in Figure \ref{fig-hexagon}. The vertex set is $V=\{x,y,z,u,v,w\}$, the edge set is $E=\{xy,yz,zu,uv,vw,wx\}$, and all edges are assigned an initial weight of $1$.

By the symmetry of the regular hexagon, the edge entropy is identical for all edges. We denote the entropy vector by $\mathbf{\Theta}^\alpha=(\Theta^\alpha_{xy},\Theta^\alpha_{yz},\Theta^\alpha_{zu},\Theta^\alpha_{uv},\Theta^\alpha_{vw},\Theta^\alpha_{wx})$, where $\Theta^\alpha_e(t)$ takes the same value for all $e\in E$. The continuous local entropy flow
$$\mathbf{w}^\prime (t)=\mathbf{\Theta}^\alpha(t),\quad \mathbf{w}(0)=\mathbf{w}_0$$
admits a unique global solution with uniform edge weights $\mathbf{w}(t)=(w(t),w(t),w(t),w(t),w(t),w(t))$, where
\begin{equation*}
w(t) = 1 + \left[ \alpha(1+\alpha)\log\frac{2}{1-\alpha} + (1-\alpha)(2-3\alpha)\log\frac{2(1-\alpha)}{\alpha} \right] t,\quad t\in[0,+\infty).
\end{equation*}
Similarly, all edges have the same cohesion value. For any edge $e\in E$, the two endpoints have one exclusive neighbor and no common neighbors. Hence
$C_e=e.$
The continuous cohesion flow
$$
\mathbf w^{\prime}(t)=\mathbf{C}(t),\qquad
\mathbf w(0)=\mathbf w_0
$$
admits the unique global solution
$\mathbf w(t)=(w(t),w(t),w(t),w(t),w(t),w(t))$,
where
$$
w(t)=1+et,\qquad t\in[0,+\infty).
$$

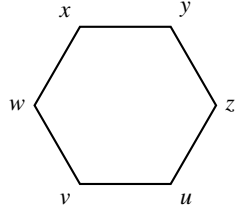
\begin{figure}[H]
\centering
\begin{tikzpicture}[scale=0.6, every node/.style={font=\small}]

\coordinate (z) at (0:2);
\coordinate (u) at (-60:2);
\coordinate (v) at (-120:2);
\coordinate (w) at (180:2);
\coordinate (x) at (120:2);
\coordinate (y) at (60:2);

\draw[thick] (x) -- (y) -- (z) -- (u) -- (v) -- (w) -- cycle;

\node[above left]  at (x) {$x$};
\node[above right] at (y) {$y$};
\node[right]       at (z) {$z$};
\node[below right] at (u) {$u$};
\node[below left]  at (v) {$v$};
\node[left]        at (w) {$w$};

\end{tikzpicture}
\caption{Regular hexagon graph}
\label{fig-hexagon}
\end{figure}
\end{example}

We next present an example of the two discrete flows.

\begin{example}\label{ex-02}
We consider a graph consisting of two squares connected by four edges, as shown in Figure \ref{fig-entropy-flow-2}. All edges are assigned an initial weight of $1$.

Take the edge $x_1x_2$ as an illustration. Its $1$-step neighborhoods are $\mathscr{N}_{x_1}=\{x_1,x_2,x_4,x_5\}$ and $\mathscr{N}_{x_2}=\{x_2,x_1,x_3,x_6\}$. Since $\mathscr{N}_{x_2}\setminus\mathscr{N}_{x_1}\neq\varnothing$, the paired local random walks are given by
\[
\mu_{x_1}^\alpha(z)=
\begin{cases}
\alpha, & z=x_1,\\[4pt]
\frac{\alpha(1-\alpha)}{3}, & z\in\{x_2,x_4,x_5\},\\[8pt]
\frac{(1-\alpha)^2}{2}, & z\in\{x_3,x_6\},
\end{cases}
\qquad
\mu_{x_2}^\alpha(z)=
\begin{cases}
\alpha, & z=x_2,\\[4pt]
\frac{\alpha(1-\alpha)}{3}, & z\in\{x_1,x_3,x_6\},\\[8pt]
\frac{(1-\alpha)^2}{2}, & z\in\{x_4,x_5\}.
\end{cases}
\]
A straightforward calculation yields the initial edge entropy
\begin{align*}
\Theta_{x_1x_2}^\alpha(0)&=\sum_{z\in \mathscr{N}_{x_1}\cup\mathscr{N}_{x_2}}\left(\mu_{x_1}^\alpha(z)\log\frac{\mu_{x_1}^\alpha(z)}{\mu_{x_2}^\alpha(z)}
+\mu_{x_2}^\alpha(z)\log\frac{\mu_{x_2}^\alpha(z)}{\mu_{x_1}^\alpha(z)}\right)\\
&=\frac{4\alpha+2\alpha^2}{3}\log\frac{3}{1-\alpha}+
\frac{2(1-\alpha)(5\alpha-3)}{3}\log\frac{2\alpha}{3(1-\alpha)}.
\end{align*}

Recall the discrete local entropy flow \eqref{discrete}, where $t_j=js$ with $s>0$ and $j=0,1,2,\cdots$. By the symmetry of the graph and the uniform initial edge weights, the initial entropy is identical for all edges, i.e., $\Theta_e^\alpha(0)=\Theta_{x_1x_2}^\alpha(0)$ for every $e\in E$. It follows that all edge weights remain equal at any time $t$:
$$w_e(t)=w_{x_1x_2}(t),\quad\forall e\in E,\,\,\forall t>0.$$
The edge entropy is constant along the flow:
$$\Theta_{e}^\alpha(t)\equiv \Theta_e^\alpha(0)=\Theta_{x_1x_2}^\alpha(0),\quad \forall e\in E,\,\,\forall t>0.$$
Therefore, the edge weight at time $t_j$ is given by
\[
w_{x_1x_2}(t_j)=1+s\sum_{k=0}^{j-1} \Theta_{x_1x_2}^\alpha(t_k)=1+js\Theta_{x_1x_2}^\alpha(0),\quad\forall j=1,2,\cdots.
\]
In particular, if we take $\alpha=0.5$, then
$$w_{e}(t_j)=1+\left(\frac{2}{3}\log 2+\log 3\right)js,\quad \forall e\in E,\,\,\forall j=1,2,\cdots.$$
Next, we compute the cohesion for each edge. 
Consider again the edge $e=x_1x_2$. We first compute the sets of exclusive and common neighbors:
\[
\mathscr{N}_{x_1}\setminus \mathscr{N}_{x_2}=\{x_4,x_5\},\quad
\mathscr{N}_{x_2}\setminus \mathscr{N}_{x_1}=\{x_3,x_6\},\quad
\mathscr{N}_{x_1x_2}=\emptyset.
\]
Hence, $A_{x_1x_2}=4$, $B_{x_1x_2}=0$,
which yields $C_{x_1x_2}=\exp(1)=e$.
By symmetry of the graph, the same calculation holds for every edge $e\in E$, and thus
\begin{equation*}
C_e(t_j)\equiv e,\quad \forall e\in E,\ \forall j\ge 0.
\end{equation*}
Consequently, the discrete cohesion flow \eqref{discrete-cohesion-flow} reduces to
\[
w_e(t_{j+1})=w_e(t_j)+s e,\qquad w_e(0)=1,
\]
which admits the explicit solution
\[
w_e(t_j)=1+jes,\quad \forall e\in E,\ \forall j\ge 0.
\]

\end{example}

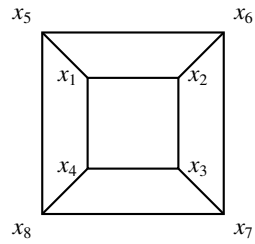
\begin{figure}[H]
\centering
\begin{tikzpicture}[scale=0.6, every node/.style={font=\small}]
\coordinate (x1) at (-1, 1);
\coordinate (x2) at ( 1, 1);
\coordinate (x3) at ( 1,-1);
\coordinate (x4) at (-1,-1);
\coordinate (x5) at (-2, 2);
\coordinate (x6) at ( 2, 2);
\coordinate (x7) at ( 2,-2);
\coordinate (x8) at (-2,-2);

\draw[thick] (x1) -- (x2) -- (x3) -- (x4) -- cycle;
\draw[thick] (x5) -- (x6) -- (x7) -- (x8) -- cycle;
\draw[thick] (x1) -- (x5);
\draw[thick] (x2) -- (x6);
\draw[thick] (x3) -- (x7);
\draw[thick] (x4) -- (x8);

\node[left] at (x1) {$x_1$};
\node[right] at (x2) {$x_2$};
\node[right] at (x3) {$x_3$};
\node[left] at (x4) {$x_4$};
\node[above left] at (x5) {$x_5$};
\node[above right] at (x6) {$x_6$};
\node[below right] at (x7) {$x_7$};
\node[below left] at (x8) {$x_8$};

\end{tikzpicture}
\caption{Two squares connected by four edges}
\label{fig-entropy-flow-2}
\end{figure}

We then demonstrate community detection via the discrete local entropy flow and the cohesion flow in the following example.

\begin{example}\label{ex-03}
Let $G=(V,E,\mathbf{w}_0)$ be a graph composed of two triangles connected by the edge $x_3x_4$, as shown in Figure \ref{fig-entropy-flow-1}. All edges are assigned an initial weight of $1$. Take $\alpha=0.5$.
The initial edge entropies corresponding to $\mathbf{w}_0$ are computed as
$\Theta_{x_3x_4}^\alpha(0) = 1.56$, $\Theta_{x_1x_2}^\alpha(0) = \Theta_{x_5x_6}^\alpha(0) = 0.35,$ $\Theta_{x_1x_3}^\alpha(0) = \Theta_{x_2x_3}^\alpha(0) = \Theta_{x_4x_5}^\alpha(0)= \Theta_{x_4x_6}^\alpha(0)= 0.93.$
Choose step size $s=0.1$ and $t_j=js$. The discrete local entropy flow is given by
\begin{equation*}
w_e(t_{j+1})=w_e(t_j)+s\Theta_e^\alpha(t_j),
\quad
w_e(0)=1.
\end{equation*}
At time $t_{10}$, the edge weights evolve to
$w_{x_3x_4}(t_{10}) = 2.47,$ $w_{x_1x_2}(t_{10}) = w_{x_5x_6}(t_{10}) = 1.42,$
$w_{x_1x_3}(t_{10}) =
w_{x_2x_3}(t_{10}) =
w_{x_4x_5}(t_{10}) =
w_{x_4x_6}(t_{10}) = 1.91.$
The initial cohesion values are given by
$C_{x_3x_4}(0)=e$, $C_{x_1x_2}(0)=C_{x_5x_6}(0)=e^{-1}$, $C_{x_1x_3}(0)=C_{x_2x_3}(0)=C_{x_4x_5}(0)=C_{x_4x_6}(0)=e^{-1/3}.$
The discrete cohesion flow is given by
\begin{equation*}
w_e(t_{j+1})=w_e(t_j)+sC_e(t_j),
\quad
w_e(0)=1.
\end{equation*}
At time $t_{10}$, the edge weights evolve to
$w_{x_3x_4}(t_{10}) = 3.72$, $w_{x_1x_2}(t_{10}) = w_{x_5x_6}(t_{10}) = 1.37,$ $w_{x_1x_3}(t_{10}) =
w_{x_2x_3}(t_{10}) =
w_{x_4x_5}(t_{10}) =
w_{x_4x_6}(t_{10}) = 1.59.$
Both flows consistently amplify the difference between inter-community and intra-community edges. In both cases, the bridge edge $x_3x_4$ exhibits the fastest growth rate, and its removal decomposes the graph into two disconnected triangles, which correspond to two natural communities.
This example demonstrates that both flows successfully identify the bridge edge connecting the two communities and therefore recover the underlying community structure.
\end{example}

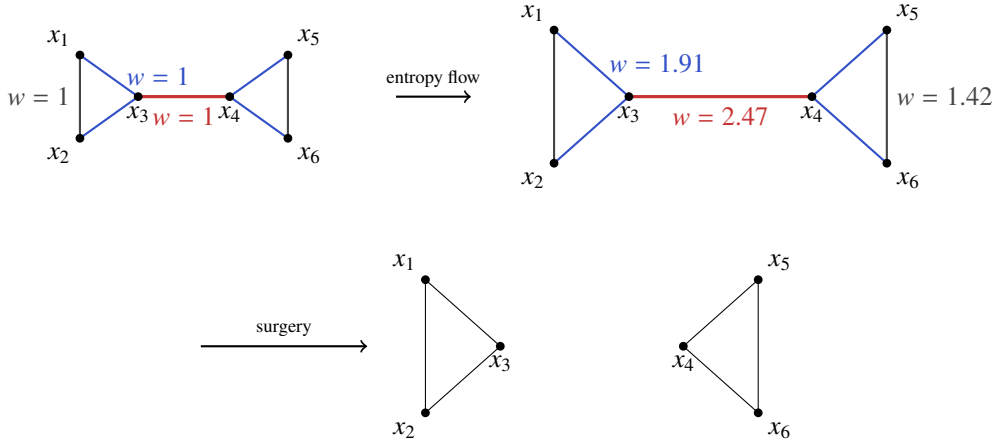
\begin{figure}[H]
\centering

\definecolor{highE}{RGB}{200,50,50}
\definecolor{midE}{RGB}{50,80,200}
\definecolor{lowE}{RGB}{70,70,70}

\begin{tikzpicture}[scale=1.1]

\begin{scope}
\coordinate (x1) at (-0.5,0.5);
\coordinate (x2) at (-0.5,-0.5);
\coordinate (x3) at (0.2,0);
\coordinate (x4) at (1.3,0);
\coordinate (x5) at (2,0.5);
\coordinate (x6) at (2,-0.5);

\draw[lowE, thick] (x1)--(x2)
    node[midway,left] {$w=1$};
\draw[lowE, thick] (x5)--(x6);

\draw[midE, thick] (x1)--(x3)
    node[midway,right=3pt] {$w=1$};
\draw[midE, thick] (x2)--(x3);
\draw[midE, thick] (x4)--(x5);
\draw[midE, thick] (x4)--(x6);

\draw[highE, very thick] (x3)--(x4)
    node[midway,below] {$w=1$};

\foreach \p in {x1,x2,x3,x4,x5,x6}
    \fill (\p) circle (1.5pt);

\node[above left] at (x1) {$x_1$};
\node[below left] at (x2) {$x_2$};
\node[below] at (x3) {$x_3$};
\node[below] at (x4) {$x_4$};
\node[above right] at (x5) {$x_5$};
\node[below right] at (x6) {$x_6$};
\end{scope}

\draw[->, thick] (3.3,0) -- (4.2,0)
    node[midway,above] {\scriptsize entropy flow};

\begin{scope}[shift={(6.2,0)}]
\coordinate (y1) at (-1,0.8);
\coordinate (y2) at (-1,-0.8);
\coordinate (y3) at (-0.1,0);
\coordinate (y4) at (2.1,0);
\coordinate (y5) at (3,0.8);
\coordinate (y6) at (3,-0.8);

\draw[lowE, thick] (y1)--(y2);
\draw[lowE, thick] (y5)--(y6)
 node[midway,right] {$w=1.42$};

\draw[midE, thick] (y1)--(y3)
    node[midway,right=3pt] {$w=1.91$};
\draw[midE, thick] (y2)--(y3);
\draw[midE, thick] (y4)--(y5);
\draw[midE, thick] (y4)--(y6);

\draw[highE, very thick] (y3)--(y4)
    node[midway,below] {$w=2.47$};

\foreach \p in {y1,y2,y3,y4,y5,y6}
    \fill (\p) circle (1.5pt);

\node[above left] at (y1) {$x_1$};
\node[below left] at (y2) {$x_2$};
\node[below] at (y3) {$x_3$};
\node[below] at (y4) {$x_4$};
\node[above right] at (y5) {$x_5$};
\node[below right] at (y6) {$x_6$};
\end{scope}

\end{tikzpicture}

\vspace{0.6cm}

\begin{tikzpicture}[scale=1.1]

\draw[->, thick] (-1.2,0) -- (0.8,0)
    node[midway,above] {\scriptsize surgery};

\begin{scope}[shift={(2.5,0)}]
\coordinate (z1) at (-1,0.8);
\coordinate (z2) at (-1,-0.8);
\coordinate (z3) at (-0.1,0);
\coordinate (z4) at (2.1,0);
\coordinate (z5) at (3,0.8);
\coordinate (z6) at (3,-0.8);

\draw (z1)--(z2) (z2)--(z3) (z3)--(z1);
\draw (z4)--(z5) (z5)--(z6) (z6)--(z4);

\foreach \p in {z1,z2,z3,z4,z5,z6}
    \fill (\p) circle (1.5pt);

\node[above left] at (z1) {$x_1$};
\node[below left] at (z2) {$x_2$};
\node[below] at (z3) {$x_3$};
\node[below] at (z4) {$x_4$};
\node[above right] at (z5) {$x_5$};
\node[below right] at (z6) {$x_6$};
\end{scope}

\end{tikzpicture}

\caption{Community detection via discrete local entropy flow}
\label{fig-entropy-flow-1}
\end{figure}

In the final example, we illustrate that the local entropy flow and the cohesion flow may yield different community partitions.

\begin{example}\label{ex-04}
Let $G=(V,E,\mathbf w_0)$ be the graph shown in Figure~\ref{fig4}. All edges are assigned initial weight $1$, and let $\alpha=0.5$. The initial values of the edge entropy are
$\Theta_{x_1y_1}^{\alpha}(0)=\Theta_{y_1z_1}^{\alpha}(0)=
\Theta_{z_1x_1}^{\alpha}(0)=1.98$,
$\Theta_{x_1x_2}^{\alpha}(0)=\Theta_{x_1x_4}^{\alpha}(0)=\Theta_{y_1y_2}^{\alpha}(0)=\Theta_{y_1y_4}^{\alpha}(0)=\Theta_{z_1z_2}^{\alpha}(0)=\Theta_{z_1z_4}^{\alpha}(0)=1.54$,
and $\Theta_{x_2x_3}^{\alpha}(0)=\Theta_{x_3x_4}^{\alpha}(0)=\Theta_{y_2y_3}^{\alpha}(0)=\Theta_{y_3y_4}^{\alpha}(0)=\Theta_{z_2z_3}^{\alpha}(0)=\Theta_{z_3z_4}^{\alpha}(0)=1.22$.
Using step size $s=0.1$, the discrete local entropy flow \eqref{discrete} yields
$w_{x_1y_1}(t_{10})=w_{y_1z_1}(t_{10})=
w_{z_1x_1}(t_{10})=2.96$,
$w_{x_1x_2}(t_{10})=w_{x_1x_4}(t_{10})=w_{y_1y_2}(t_{10})=w_{y_1y_4}(t_{10})=w_{z_1z_2}(t_{10})=w_{z_1z_4}(t_{10})=2.55$,
and $w_{x_2x_3}(t_{10})=w_{x_3x_4}(t_{10})=w_{y_2y_3}(t_{10})=w_{y_3y_4}(t_{10})=w_{z_2z_3}(t_{10})=w_{z_3z_4}(t_{10})=2.23$.
Deleting the three edges with the largest weights, namely $x_1y_1$, $y_1z_1$, and $z_1x_1$, the graph $G$ is partitioned into three connected components.

The corresponding initial values of the cohesion are
$C_{x_1y_1}(0)=C_{y_1z_1}(0)=
C_{z_1x_1}(0)=1.40$, while all remaining edges satisfy
$C_e(0)=2.72.$
Using the same step size $s=0.1$, after the discrete cohesion flow
\eqref{discrete-cohesion-flow},
at time $t_{10}$, the evolved edge weights are
$w_{x_1y_1}(t_{10})=w_{y_1z_1}(t_{10})=
w_{z_1x_1}(t_{10})=2.21$, while all remaining edges satisfy
$w_e(t_{10})=3.72.$
Deleting the three edges with the largest  weights produces a community partition different from that obtained by the local entropy flow. This example shows that the local entropy flow and the cohesion flow may rank edges differently, and hence may lead to different community detection results.
\end{example}

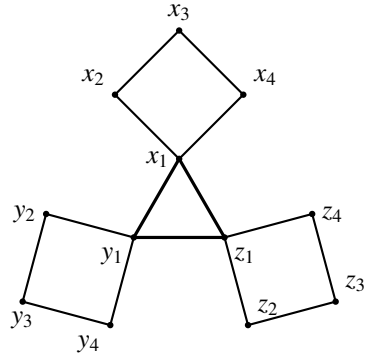
\begin{figure}[H]
\centering

\def\L{2}
\pgfmathsetmacro{\H}{\L*sqrt(3)/2}
\pgfmathsetmacro{\d}{\L/sqrt(2)}

\begin{tikzpicture}[scale=0.6, thick]

\coordinate (A) at (0,\H);
\coordinate (B) at (-\L/2,0);
\coordinate (C) at (\L/2,0);

\draw[very thick] (A) -- (B);
\draw[very thick] (B) -- (C);
\draw[very thick] (C) -- (A);

\begin{scope}[shift=(A), rotate=0]
    \draw[thick] (0,0) -- (-\d,\d);
    \draw[thick] (-\d,\d) -- (0,2*\d);
    \draw[thick] (0,2*\d) -- (\d,\d);
    \draw[thick] (\d,\d) -- (0,0);
    \fill (-\d,\d) circle (2pt);
    \fill (0,2*\d) circle (2pt);
    \fill (\d,\d) circle (2pt);
    \node[above left] at (-\d,\d) {$x_2$};
    \node[above] at (0,2*\d) {$x_3$};
    \node[above right] at (\d,\d) {$x_4$};
\end{scope}

\begin{scope}[shift=(B), rotate=120]
    \draw[thick] (0,0) -- (-\d,\d);
    \draw[thick] (-\d,\d) -- (0,2*\d);
    \draw[thick] (0,2*\d) -- (\d,\d);
    \draw[thick] (\d,\d) -- (0,0);
    \fill (-\d,\d) circle (2pt);
    \fill (0,2*\d) circle (2pt);
    \fill (\d,\d) circle (2pt);
    \node[below left] at (-\d,\d) {$y_4$};
    \node[below] at (0,2*\d) {$y_3$};
    \node[left] at (\d,\d) {$y_2$};
\end{scope}

\begin{scope}[shift=(C), rotate=240]
    \draw[thick] (0,0) -- (-\d,\d);
    \draw[thick] (-\d,\d) -- (0,2*\d);
    \draw[thick] (0,2*\d) -- (\d,\d);
    \draw[thick] (\d,\d) -- (0,0);
    \fill (-\d,\d) circle (2pt);
    \fill (0,2*\d) circle (2pt);
    \fill (\d,\d) circle (2pt);
    \node[right] at (-\d,\d) {$z_4$};
    \node[above right] at (0,2*\d) {$z_3$};
    \node[above right] at (\d,\d) {$z_2$};
\end{scope}

\fill (A) circle (2pt) node[left] {$x_1$};
\fill (B) circle (2pt) node[below left] {$y_1$};
\fill (C) circle (2pt) node[below right] {$z_1$};

\end{tikzpicture}
\caption{Three squares connected by a triangle}
\label{fig4}
\end{figure}

\section{Proofs of the main theorems}\label{theorems}

In this section, we present the proofs of the existence, uniqueness, and convergence of solutions to the continuous local entropy flow \eqref{entropy-flow} and the discrete local entropy flow \eqref{discrete}. All proofs are established based on classical ordinary differential equation (ODE) theory. We also provide the corresponding proofs for the cohesion flow \eqref{cohesion-flow} and the discrete cohesion flow \eqref{discrete-cohesion-flow}.

\medskip
\noindent\textbf{Proof of Theorem \ref{existence}}.
We follow the proof framework of Theorem 2.1 in our previous work \cite{entropyflow}, with adjustments adapted to the local random walk setting. Let $\alpha\in(0,1)$, and for any edge $e=xy\in E$, let $\Theta_e^\alpha$ be the edge entropy defined in \eqref{theta}. The proof is divided into three steps.

\medskip
\noindent\textit{Step 1: Short-time existence and uniqueness of solutions}

In view of \eqref{greater}, we have
\begin{equation}\label{claim-1}
\Theta_e^\alpha\geq 0
\end{equation}
for all $e\in E$, with equality $\Theta_e^\alpha=0$ if and only if $\mu_x^\alpha(z)=\mu_y^\alpha(z)$ for all $z\in\mathscr{N}_x\cup\mathscr{N}_y$.

Let $E=\{e_1,e_2,\dots,e_m\}$, and denote
$$\mathbb{R}_+^m = \{\mathbf{w}=(w_1,w_2,\dots,w_m)\in\mathbb{R}^m: w_j>0 \text{ for all } 1\leq j\leq m\}.$$
Define the map $\mathbf{F}:\mathbb{R}_+^m\rightarrow\mathbb{R}^m$ by
$$\mathbf{F}(\mathbf{w})=(F_1(\mathbf{w}),F_2(\mathbf{w}),\dots,F_m(\mathbf{w}))^\top,$$
where $F_j(\mathbf{w})=\Theta_{e_j}^\alpha$ is the edge entropy of $e_j$ computed with respect to the weight vector $\mathbf{w}$.

By the Picard-Lindel\"of theorem, it suffices to show that $\mathbf{F}$ is locally Lipschitz continuous on $\mathbb{R}_+^m$. For any compact subset $\overline{\Omega}\subset\mathbb{R}_+^m$, and any two weight vectors $\mathbf{w},\widetilde{\mathbf{w}}\in\overline{\Omega}$, let $\mu_x^\alpha$ and $\widetilde{\mu}_x^\alpha$ be the local random walks induced by $\mathbf{w}$ and $\widetilde{\mathbf{w}}$, respectively. For each edge $e_j=x_jy_j\in E$, there exists a constant $C>0$, depending only on $m$, $e_j$, and the distance between $\overline{\Omega}$ and the boundary $\partial\mathbb{R}_+^m$, such that
$$
|\mu_{x_j}^\alpha(z)-\widetilde{\mu}_{x_j}^\alpha(z)|+|\mu_{y_j}^\alpha(z)-\widetilde{\mu}_{y_j}^\alpha(z)|\leq C\|\mathbf{w}-\widetilde{\mathbf{w}}\|
$$
for all $z\in\mathscr{N}_{x_j}\cup\mathscr{N}_{y_j}$. A straightforward calculation further yields
$$|F_j(\mathbf{w})-F_j(\widetilde{\mathbf{w}})|\leq C\|\mathbf{w}-\widetilde{\mathbf{w}}\|$$
for all $1\leq j\leq m$, which confirms that $\mathbf{F}$ is local Lipschitz on $\mathbb{R}_+^m$. Thus there exists $T_0>0$ such that the local entropy flow \eqref{entropy-flow} admits a unique solution $\mathbf{w}(t)=(w_{e_1}(t),w_{e_2}(t),\dots,w_{e_m}(t))$ on the time interval $[0,T_0]$.

\medskip
\noindent\textit{Step 2: Global existence of solutions}

Let
$$T^\ast=\sup\left\{T>0: \text{the flow \eqref{entropy-flow} admits a unique solution on } [0,T]\right\}.$$
We prove $T^\ast=+\infty$ by contradiction. Suppose $T^\ast<+\infty$. Then the unique solution $\mathbf{w}(t)$ of the initial value problem
$$
\begin{cases}
\mathbf{w}^\prime(t)=\mathbf{F}(\mathbf{w}(t)),& t\in[0,T^\ast),\\
w_e(t)>0,& \forall e\in E,\, t\in[0,T^\ast),\\
w_e(0)=w_{0,e},& \forall e\in E
\end{cases}
$$
cannot be extended beyond $T^\ast$.

From the non-negativity of edge entropy \eqref{claim-1}, we have $w_e^\prime(t)=\Theta_e^\alpha(t)\geq 0$ for all $e\in E$ and $t\in[0,T^\ast)$. Thus
$$w_e(t)\geq w_{0,e}>0,\quad \forall t\in[0,T^\ast),\,\forall e\in E.$$
Combined with the definition of $\Theta_e^\alpha$, this implies
$$\Theta_e^\alpha(t)\leq C\left(1+\log\sum_{\tau\in E}w_\tau(t)\right)$$
for some constant $C>0$ depending only on $m$, $n$, and the initial weight $\mathbf{w}_0$. Substituting this into the flow equation \eqref{entropy-flow}, we obtain
$$
\frac{d}{dt}\left(\sum_{\tau\in E}w_\tau(t)\right)\leq C\left(1+\log\sum_{\tau\in E}w_\tau(t)\right)\leq C\left(1+\sum_{\tau\in E}w_\tau(t)\right).
$$
By Gr\"onwall's inequality, we have
$$
0<\min_{\tau\in E}w_{0,\tau}\leq w_e(t)\leq\sum_{\tau\in E}w_\tau(t)\leq Ce^{Ct},\quad \forall t\in[0,T^\ast).
$$
This shows that every component of $\mathbf{w}(t)$ is uniformly bounded from above, and admits a uniform positive lower bound on $[0,T^\ast)$. By the ODE extension theorem, the solution $\mathbf{w}(t)$ can be extended to $[0,T_1)$ for some $T_1>T^\ast$, which contradicts the definition of $T^\ast$. Hence $T^\ast=+\infty$, and the flow admits a unique global solution on $[0,+\infty)$.

\medskip
\noindent\textit{Step 3: Long-time convergence of the solution}

Since $w_\tau^\prime(t)=\Theta_\tau^\alpha(t)\geq 0$ for all $t\geq0$, for each edge $\tau\in E$, $w_\tau(t)$ is non-decreasing in $t$. Thus, for every $\tau\in E$, either $w_\tau(t)\to+\infty$ or $w_\tau(t)\to w_\tau^\ast\geq w_{0,\tau}>0$ as $t\to+\infty$. We now prove that exactly one of the following two alternatives holds:
\begin{enumerate}
    \item[(i)] $w_\tau(t)\to+\infty$ for all $\tau\in E$;
    \item[(ii)] There exists $w_\tau^\ast>0$ such that $w_\tau(t)\to w_\tau^\ast$ for all $\tau\in E$.
\end{enumerate}

We proceed by contradiction. Suppose that neither (i) nor (ii) holds. Since $G$ is connected, there must exist two edges sharing a common vertex, say $\tau_1=xy$ and $\tau_2=xz$, such that
\begin{equation}\label{contrary-1}
w_{\tau_1}(t)\to w_{\tau_1}^\ast<+\infty, \quad \text{as } t\to+\infty,
\end{equation}
and
\begin{equation}\label{contrary-2}
w_{\tau_2}(t)\to +\infty, \quad \text{as } t\to+\infty.
\end{equation}
Recall that the local random walk starting from $x$ satisfies
$$
\mu_x^\alpha(y)=
\begin{cases}
(1-\alpha)\frac{w_{xy}}{\sum\limits_{u\sim x}w_{xu}}, & \text{if } \mathscr{N}_y\setminus\mathscr{N}_x=\varnothing,\\[1.5ex]
\alpha(1-\alpha)\frac{w_{xy}}{\sum\limits_{u\sim x}w_{xu}}, & \text{if } \mathscr{N}_y\setminus\mathscr{N}_x\neq\varnothing,
\end{cases}
$$
and $\mu_y^\alpha(y)=\alpha$ by definition. For any $t\geq0$, we have
\begin{align*}
\mu_y^\alpha(y)\log\frac{\mu_y^\alpha(y)}{\mu_x^\alpha(y)}
&\geq\alpha\log\frac{\alpha\sum_{u\sim x}w_{xu}(t)}{(1-\alpha)w_{xy}(t)}\\
&\geq\alpha\log\frac{\alpha}{1-\alpha}+\alpha\log\frac{w_{\tau_2}(t)}{w_{\tau_1}(t)}.
\end{align*}
Combined with the universal lower bound
$$\mu_y^\alpha(y)\log\frac{\mu_y^\alpha(y)}{\mu_x^\alpha(y)}\geq\mu_y^\alpha(y)\log\mu_y^\alpha(y)\geq -e^{-1},$$
we obtain
\begin{equation}\label{ggg}
\Theta_{\tau_1}^\alpha(t)\geq \alpha\log\frac{w_{\tau_2}(t)}{w_{\tau_1}(t)}-C
\end{equation}
for some constant $C>0$ depending only on $\alpha$ and $n$. From \eqref{contrary-1}, \eqref{contrary-2} and \eqref{ggg}, there exists $T>0$ such that
$$\Theta_{\tau_1}^\alpha(t)\geq 1,\quad\forall t\geq T.$$
As a consequence, we have
$$w_{\tau_1}(t)\geq w_{\tau_1}(T)+\int_T^t\Theta_{\tau_1}^\alpha(s)ds\geq t-T.$$
Letting $t\to+\infty$, we get $w_{\tau_1}(t)\to+\infty$, which contradicts \eqref{contrary-1}. Hence either (i) or (ii) must hold.

It remains to show that $\lim_{t\to+\infty}\Theta_e^\alpha(t)=0$ for all $e\in E$ in case (ii). For each edge $\tau$, since $w_\tau(t)\to w_\tau^\ast>0$, the edge entropy converges to $\Theta_\tau^\ast$, which is the entropy of $\tau$ with respect to the steady weight $\mathbf{w}^\ast=(w_\tau^\ast)_{\tau\in E}$. Suppose $\Theta_\tau^\ast\neq0$ for some $\tau$. Then $\Theta_\tau^\ast>0$ by \eqref{claim-1}, and by the mean value theorem,
$$w_\tau(t+1)-w_\tau(t)=w_\tau^\prime(\xi_t)=\Theta_\tau^\alpha(\xi_t)\to\Theta_\tau^\ast>0,\quad\text{as }t\to+\infty,$$
where $\xi_t\in[t,t+1]$. This contradicts the convergence of $w_\tau(t)$ to a finite limit $w_\tau^\ast$. Hence $\Theta_\tau^\ast=0$ for all $\tau\in E$, which completes the proof.
\hfill$\Box$

\medskip
\noindent\textbf{Proof of Theorem \ref{discrete-existence}}.
Since the discrete flow is an explicit forward iterative scheme, the existence and uniqueness of the solution to \eqref{discrete} hold trivially for all $j=0,1,2,\cdots$.

We now focus on the long-time convergence of the solution and the proof is similar to that of Theorem \ref{existence}. For each edge $\tau\in E$, since $w_\tau(t_{j+1})-w_\tau(t_j)=s\Theta_\tau^\alpha(t_j)\geq0$ for all $j\geq0$, the weight sequence $\{w_\tau(t_j)\}_{j\geq0}$ is non-decreasing. Therefore, for every $\tau\in E$, either $w_\tau(t_j)\to+\infty$ or $w_\tau(t_j)\to w_\tau^\ast\geq w_{0,\tau}>0$ as $j\to+\infty$. It remains to prove that exactly one of the following two alternatives holds:
\begin{equation}\label{fin}
w_\tau(t_j)\to w_\tau^\ast<+\infty\quad \text{for all } \tau\in E,
\end{equation}
or
\begin{equation}\label{infin}
w_\tau(t_j)\to +\infty\quad \text{for all } \tau\in E.
\end{equation}

Suppose neither \eqref{fin} nor \eqref{infin} holds. Since $G$ is connected, there must exist two edges sharing a common vertex, say $\tau_1=xy$ and $\tau_2=xz$, such that
$$w_{\tau_1}(t_j)\to w_{\tau_1}^\ast<+\infty, \quad \text{as } j\to+\infty,$$
and
$$w_{\tau_2}(t_j)\to +\infty, \quad \text{as } j\to+\infty.$$
Note that the lower bound inequality \eqref{ggg} derived in the proof of Theorem \ref{existence} relies only on the definition of local random walks and the positivity of edge weights, and thus holds equally for the discrete setting at every iteration $t_j$. That is,
$$\Theta_{\tau_1}^\alpha(t_j)\geq \alpha\log\frac{w_{\tau_2}(t_j)}{w_{\tau_1}(t_j)}-C$$
for some constant $C>0$ depending only on $\alpha$ and $n$. Since $w_{\tau_2}(t_j)\to+\infty$ and $w_{\tau_1}(t_j)\to w_{\tau_1}^\ast<+\infty$, there exists a positive integer $N$ such that
$$\Theta_{\tau_1}^\alpha(t_j)\geq 1,\quad \forall j\geq N.$$
For all $j>N$, substituting this into the discrete flow update rule yields
$$w_{\tau_1}(t_j)=w_{\tau_1}(t_N)+s\sum_{k=N}^{j-1}\Theta_{\tau_1}^\alpha(t_k)\geq s(j-N).$$
Letting $j\to+\infty$, we obtain $w_{\tau_1}(t_j)\to+\infty$, which contradicts the assumption $w_{\tau_1}(t_j)\to w_{\tau_1}^\ast<+\infty$. Hence either \eqref{fin} or \eqref{infin} must hold.

Finally, we prove that $\lim\limits_{j\to+\infty}\Theta_\tau^\alpha(t_j)=0$ for all $\tau\in E$ in the case of \eqref{fin}. For each edge $\tau$, if $w_\tau(t_j)\to w_\tau^\ast<+\infty$ as $j\to+\infty$, then the difference of the convergent sequence satisfies
$$\lim_{j\to+\infty}\left(w_\tau(t_{j+1})-w_\tau(t_j)\right)=0.$$
By the discrete flow update rule, $w_\tau(t_{j+1})-w_\tau(t_j)=s\Theta_\tau^\alpha(t_j)$ with $s>0$. It follows immediately that
$$\lim_{j\to+\infty}\Theta_\tau^\alpha(t_j)=0,$$
which completes the proof.
\hfill$\Box$

\medskip
\noindent\textbf{Proof of Theorem \ref{cohesion-existence}.}
The cohesion $C_e=\exp(\widetilde C_{xy})$ is a composition of rational functions and the exponential function, hence continuously differentiable on $\mathbb{R}_+^m$. Consequently, the vector field of the ODE system \eqref{cohesion-flow} is local Lipschitz  on $\mathbb{R}_+^m$. By the Picard-Lindel\"of theorem, the cohesion flow \eqref{cohesion-flow} admits a unique solution on $[0,T^\ast)$.

Since $-1\le \widetilde C_{xy}\le 1$, we have $e^{-1}\le C_e\le e$, and hence
$$e^{-1}\le w_e^\prime(t)\le e,\qquad \forall e\in E.$$
Integrating from $0$ to $t$ yields
$$w_{0,e}+e^{-1}t\le w_e(t)\le w_{0,e}+et,\qquad \forall t\ge0.$$
By the ODE extension theorem, the solution can be extended  $T^\ast=+\infty$. The lower bound further implies that each $w_e(t)$ is strictly increasing and $w_e(t)\to+\infty$. This completes the proof. \hfill$\Box$

\medskip
\noindent\textbf{Proof of Theorem \ref{discrete-cohesion-existence}.}
The discrete cohesion flow is defined by the explicit recurrence
\[
w_e(t_{j+1})=w_e(t_j)+sC_e(t_j),\qquad j=0,1,2,\cdots
\]
Given the initial weights $w_e(0)=w_{0,e}>0$, the right-hand side is uniquely determined at each step, so the sequence $\{w_e(t_j)\}_{j\ge0}$ exists and is unique for all $j$.
Since $e^{-1}\le C_e(t_j)\le e$ for every edge $e$, it follows that
\[
se^{-1}\le w_e(t_{j+1})-w_e(t_j)\le se.
\]
Summing these inequalities from $k=0$ to $j-1$ yields
\[
w_{0,e}+je^{-1}s\;\le\;w_e(t_j)\;\le\;w_{0,e}+jes,\qquad j\ge0.
\]
Thus, $\{w_e(t_j)\}_{j\ge0}$ is strictly increasing for all $j$, and the lower bound implies $w_e(t_j)\to+\infty$ as $j\to\infty$, completing the proof. \hfill$\Box$

\section{Applications and numerical experiments}\label{experimental}

In this section, we evaluate the proposed local entropy flow (LEF) and cohesion flow (CF) on community detection and node classification. We denote the global entropy flow from our previous work \cite{entropyflow} as GEF, and the combination of CF and LEF as CLEF. The proposed graph evolution mechanisms are implemented through the discrete local entropy flow \eqref{discrete} and the discrete cohesion flow \eqref{discrete-cohesion-flow}, which serve as numerical realizations of the continuous flows \eqref{entropy-flow} and \eqref{cohesion-flow}.
Theorems \ref{discrete-existence} and \ref{discrete-cohesion-existence} guarantee that these iterative schemes admit unique solutions for all iterations.

For community detection, we first examine three strategies: LEF alone, CF alone, and their sequential combination CLEF.
We then compare CLEF with several classical  algorithms and Ricci curvature-based methods.
In addition, to evaluate the role of local entropy in the two-stage process, 
we consider CGEF (CF followed by GEF) as a comparison baseline.
For node classification, we adopt the GEGCN architecture \cite{GEGCN} and replace its Ricci flow-based graph evolution with our entropy-based mechanisms. We evaluate LEF and CLEF against a range of baseline methods on both homophilic and heterophilic datasets. Results for GEF and CGEF are also reported to compare local and global entropy evolution.

Overall, LEF and CF are both effective as  mechanisms for graph evolution. CLEF yields more significant improvements for community detection, whereas the differences among entropy-based variants are relatively minor for node classification.

\subsection{Community detection}\label{sec:community-detection}

Community detection aims to partition a graph into densely connected groups with sparse inter-group connections, with broad applications across disciplines \cite{Bhowmick, GN, Tauro}. Classical approaches include modularity optimization, spectral clustering, and probabilistic graphical models \cite{Newman, Fortunato, GN, Newman2006}. Recently, discrete Ricci curvature and Ricci flow have emerged as effective geometric tools for graph analysis \cite{LaiX, Ma1, NLLG, Tian-Ma-Yang-Zhao}, and our previous work introduced GEF as an information-theoretic alternative \cite{entropyflow}.

\subsubsection{Algorithm and experimental setup}
We evaluate three strategies for community detection: (i) LEF alone, (ii) CF alone, and (iii) CLEF (CF + LEF). All strategies follow the same general procedure: evolve edge weights according to the corresponding discrete flow, then remove edges with large weights and treat the remaining connected components as communities.
The CLEF pipeline is outlined in Algorithm~\ref{alg:community-detection}, the other strategies follow the same edge evolution and removal procedure using their respective flow rules.

\begin{algorithm}[h]
\caption{Community detection via CLEF (CF + LEF)}
\label{alg:community-detection}
\KwIn{Connected graph $G=(V, E, \mathbf{w}_0)$; parameters $\alpha \in (0,1)$; step sizes $s_c, s_e > 0$; number of cohesion iterations  $N_c$; number of entropy iterations $N_e$.}
\KwOut{Community partition of $G$.}

\textbf{Step 1: Cohesion-driven edge evolution}

\For{$i = 1,2,\dots,N_c$}{
    Compute edge cohesion $C_e(t_{i-1})$ for all $e \in E$\;
    Update edge weights:
    $w_e(t_i) \leftarrow w_e(t_{i-1}) + s_c \cdot C_e(t_{i-1})$\;
}

\textbf{Step 2: Entropy-driven edge evolution}

\For{$j = 1,2,\dots,N_e$}{
    Construct local random walk distributions for all nodes in $G$\;
    Compute edge entropy $\Theta_e(t_{j-1})$ for all $e \in E$\;
    Update edge weights:
    $w_e(t_j) \leftarrow w_e(t_{j-1}) + s_e \cdot \Theta_e(t_{j-1})$\;
}

\textbf{Step 3: Community extraction}

\For{$\theta = \max\limits_{e \in E} w_e,\ \ldots,\ \min\limits_{e \in E}$}{
    Remove all edges with $w_e \geq \theta$ from $E$\;
    Extract connected components as candidate communities\;
    Evaluate the quality of the current partition\;
}

\Return the partition with the best quality score.
\end{algorithm}

We evaluate all methods on three widely used real-world benchmark networks: Karate \cite{karate}, Football \cite{GN}, and Facebook \cite{facebook}. We adopt three evaluation metrics: Adjusted Rand Index (ARI), Normalized Mutual Information (NMI), and Modularity (Q) \cite{NMI, ARI, Q}.

We compare our methods with 8 representative baseline approaches:
\begin{itemize}
    \item Classical methods: Girvan-Newman \cite{GN}, greedy modularity maximization \cite{Newman, Reichardt}, and label propagation \cite{Cordasco};
    \item Ricci curvature-based methods: DORF \cite{NLLG}, NDORF, NDSRF \cite{LaiX}, and two variants Rho and RhoN \cite{Ma1}.
\end{itemize}

Hyperparameters are selected via empirical analysis. For all strategies, we fix the step size $s=0.1$ and set $\alpha=0.5$ for the entropy-based flows. 
For CF, we set the number of iterations to $N_c=30$ for Karate, $N_c=5$ for Football, and $N_c=30$ for Facebook. For LEF, we set $N_e=10$ uniformly across all three datasets.
For CLEF, we set $N_c=30$ and $N_e=40$ on Karate, $N_c=30$ and $N_e=10$ on Football, and $N_c=30$ and $N_e=10$ on Facebook. For CGEF,  $N_g$ denotes the number of iterations of the global entropy flow.
We set $N_c=10$ and $N_g=20$ on Karate, $N_c=10$ and $N_g=10$ on Football, and $N_c=20$ and $N_g=11$ on Facebook.

\subsubsection{Experimental results and analysis}

We first examine the performance of the three proposed strategies. As shown in Table~\ref{tab:community_proposed}, both LEF and CF are effective on certain datasets, but their performance is not consistent across all cases. Specifically, LEF performs reasonably well on Karate (ARI=0.48) and achieves strong modularity on Facebook (Q=0.93), but it performs poorly on Football (ARI 0.03). In contrast, CF achieves strong results on Football (ARI=0.93, NMI=0.94), but its performance is weaker on Karate (ARI=0.16) and Facebook (ARI=0.67). These results indicate that each flow captures different structural characteristics of the network, but neither provides uniformly strong performance across all datasets.

In contrast, the two-stage combination CLEF achieves consistently strong performance across all datasets. On Karate, it attains the best ARI (0.83) and NMI (0.78). On Football, it matches the best-performing methods (ARI=0.93, NMI=0.94). On Facebook, it achieves the highest ARI (0.72) and competitive modularity. These results confirm that the sequential combination of CF and LEF effectively integrates their complementary strengths, yielding robust community detection performance.

\begin{table}[htbp]
\centering
\caption{Performance of the three proposed strategies}
\label{tab:community_proposed}
\small
\begin{tabular}{lccc|ccc|ccc}
\toprule
 & \multicolumn{3}{c|}{Karate} & \multicolumn{3}{c|}{Football} & \multicolumn{3}{c}{Facebook} \\
\cmidrule(lr){2-4} \cmidrule(lr){5-7} \cmidrule(lr){8-10}
Methods & ARI & NMI & Q & ARI & NMI & Q & ARI & NMI & Q \\
\midrule
LEF & 0.48 & 0.55 & \first{0.82} & 0.03 & 0.68 & 0.50 & 0.04 & 0.51 & \first{0.93} \\
CF & 0.16 & 0.39 & 0.70 & \first{0.93} & \first{0.94} & 0.91 & 0.67 & \first{0.70} & 0.86 \\
CLEF & \first{0.83} & \first{0.78} & 0.80 & \first{0.93} & \first{0.94} & \first{0.92} & \first{0.72} & \first{0.70} & \first{0.93} \\
\bottomrule
\end{tabular}
\end{table}

We further compare CLEF with 8 baseline methods. Table~\ref{tab:community_full} reports the performance of all methods on the three datasets. On Karate, CLEF achieves an ARI of 0.83 and an NMI of 0.78, which are higher than those obtained by all baseline methods. Its modularity is 0.80, lower than RhoN (0.84) but comparable to several other methods. On Football, CLEF achieves an ARI of 0.93, an NMI of 0.94, and a modularity of 0.92, matching the best results among the baseline methods. On Facebook, CLEF achieves the best ARI (0.72) among all methods, while its NMI (0.70) and modularity (0.93) are slightly lower than those of RhoN (0.72 and 0.95, respectively). Overall, these results suggest that the proposed two-stage framework performs competitively across the three benchmark networks.

\begin{table}[H]
\centering
\caption{Comparison of CLEF with baseline methods}
\label{tab:community_full}
\small
\begin{tabular}{lccc|ccc|ccc}
\toprule
 & \multicolumn{3}{c|}{Karate} & \multicolumn{3}{c|}{Football} & \multicolumn{3}{c}{Facebook} \\
\cmidrule(lr){2-4} \cmidrule(lr){5-7} \cmidrule(lr){8-10}
Methods & ARI & NMI & Q & ARI & NMI & Q & ARI & NMI & Q \\
\midrule
Girvan-Newman & \second{0.77} & \second{0.73} & 0.48 & 0.14 & 0.36 & 0.50 & 0.03 & 0.16 & 0.01 \\
Greedy Modularity & 0.57 & 0.56 & 0.58 & 0.47 & 0.70 & 0.82 & 0.49 & 0.68 & 0.55 \\
Label Propagation & 0.38 & 0.36 & 0.54 & 0.75 & 0.87 & 0.90 & 0.39 & 0.65 & 0.51 \\
DORF & 0.59 & 0.57 & 0.69 & \first{0.93} & \first{0.94} & \second{0.91} & 0.67 & \first{0.73} & 0.68 \\
NDORF & 0.59 & 0.57 & 0.69 & \first{0.93} & \first{0.94} & \second{0.91} & 0.68 & \first{0.73} & 0.68 \\
NDSRF & 0.59 & 0.57 & 0.68 & \first{0.93} & \first{0.94} & \second{0.91} & 0.68 & \first{0.73} & 0.68 \\
Rho & \second{0.77} & 0.68 & \second{0.82} & \second{0.89} & 0.92 & 0.90 & 0.64 & \second{0.72} & 0.63 \\
RhoN & \second{0.77} & 0.68 & \first{0.84} & \second{0.89} & \second{0.93} & \first{0.92} & \second{0.69} & \second{0.72} & \second{0.95} \\
\midrule
CLEF & \first{0.83} & \first{0.78} & 0.80 & \first{0.93} & \first{0.94} & \first{0.92} & \first{0.72} & 0.70 & 0.93 \\
\bottomrule
\end{tabular}
\end{table}

In addition, we report the performance of CGEF (CF followed by GEF) as a supplementary comparison. Table~\ref{tab:clef_vs_cgef} shows that CLEF and CGEF achieve nearly identical results across all three datasets, indicating that the choice between local and global entropy flow in the second stage has minimal impact on performance.

\begin{table}[htbp]
\centering
\caption{Comparison of CLEF and CGEF}
\label{tab:clef_vs_cgef}
\small
\begin{tabular}{lccc|ccc|ccc}
\toprule
 & \multicolumn{3}{c|}{Karate} & \multicolumn{3}{c|}{Football} & \multicolumn{3}{c}{Facebook} \\
\cmidrule(lr){2-4} \cmidrule(lr){5-7} \cmidrule(lr){8-10}
Methods & ARI & NMI & Q & ARI & NMI & Q & ARI & NMI & Q \\
\midrule
CLEF & 0.83 & 0.78 & 0.80 & 0.93 & 0.94 & 0.92 & 0.72 & 0.70 & 0.93 \\
CGEF & 0.83 & 0.78 & 0.81 & 0.93 & 0.94 & 0.92 & 0.72 & 0.71 & 0.94 \\
\bottomrule
\end{tabular}
\end{table}

To further illustrate the results, we visualize the community partitions obtained by CLEF on the three benchmark networks in Figure \ref{fig:entropy_flow_visualization}. The proposed method successfully identifies community structures on all three networks. Densely connected regions are assigned to the same community, while inter-community bridge edges are effectively removed during the flow process, resulting in clear community partitions.

\begin{figure}[H]
    \centering
    \begin{subfigure}{0.28\textwidth}
        \centering
        \includegraphics[width=\textwidth]{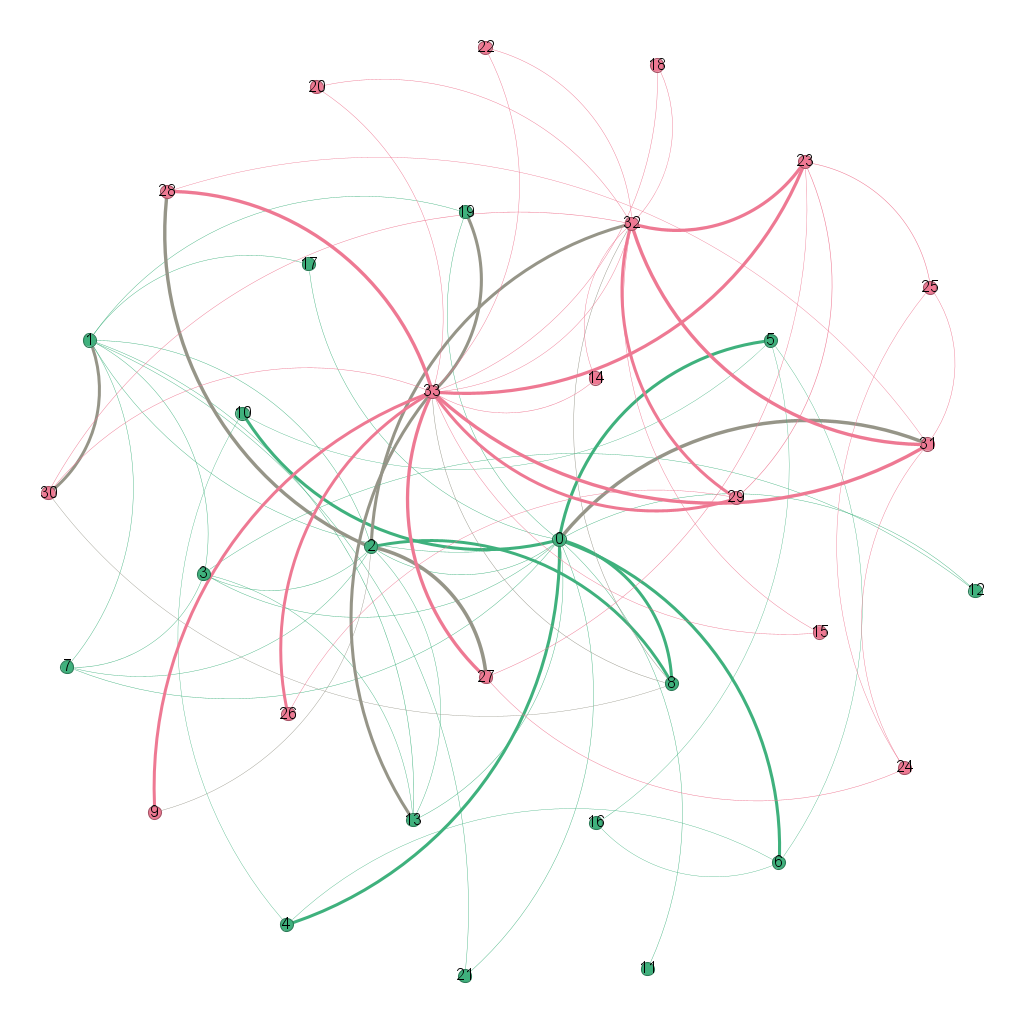}
        \caption{Karate (Original)}
    \end{subfigure}
    \hfill
    \begin{subfigure}{0.28\textwidth}
        \centering
        \includegraphics[width=\textwidth]{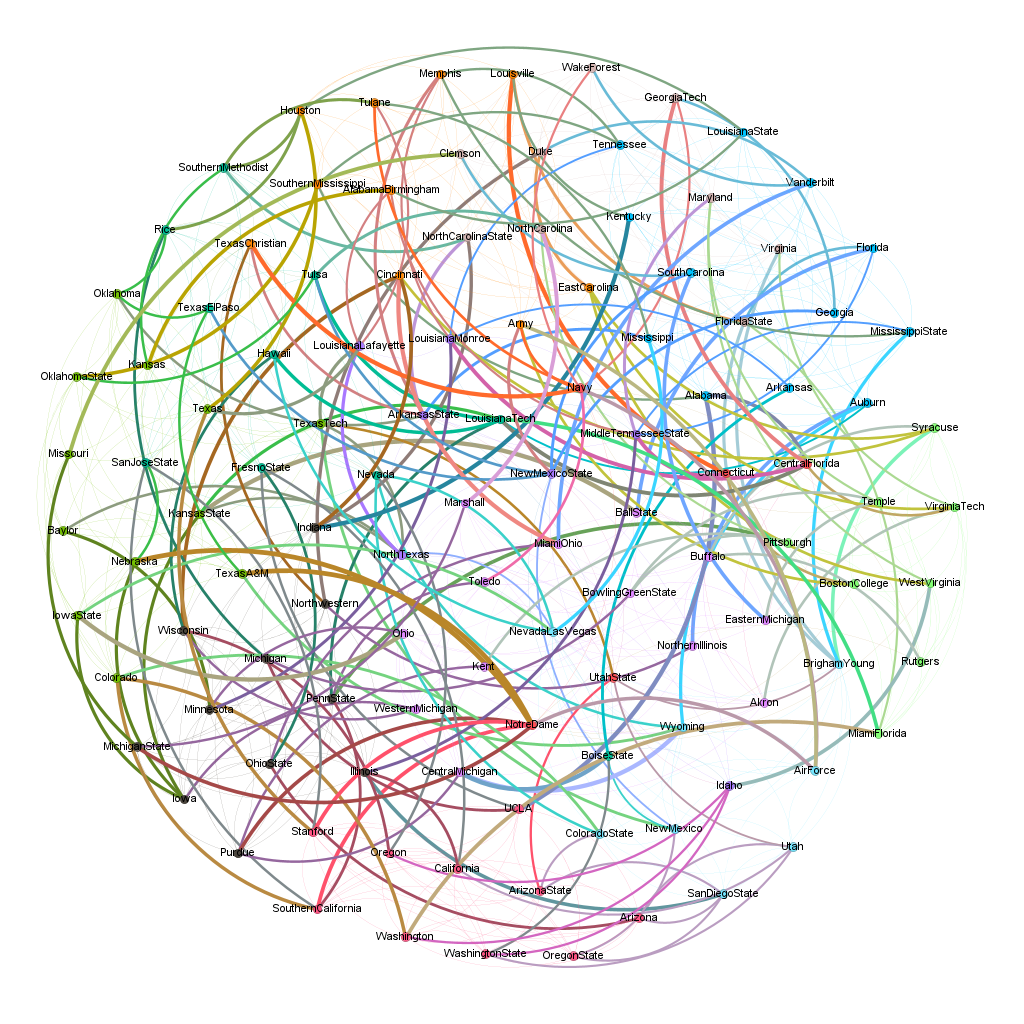}
        \caption{Football (Original)}
    \end{subfigure}
    \hfill
    \begin{subfigure}{0.28\textwidth}
        \centering
        \includegraphics[width=\textwidth]{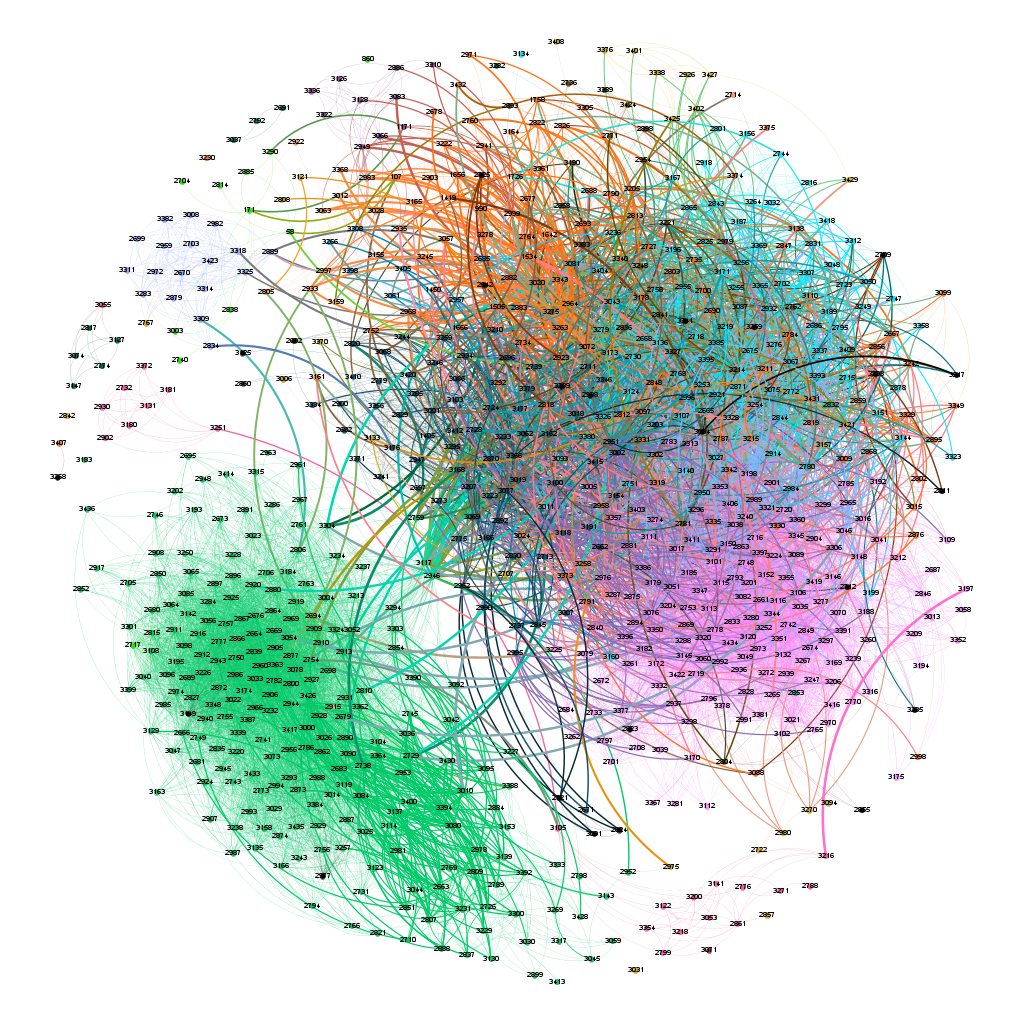}
        \caption{Facebook (Original)}
    \end{subfigure}

    \vspace{0.3cm}

    \begin{subfigure}{0.28\textwidth}
        \centering
        \includegraphics[width=\textwidth]{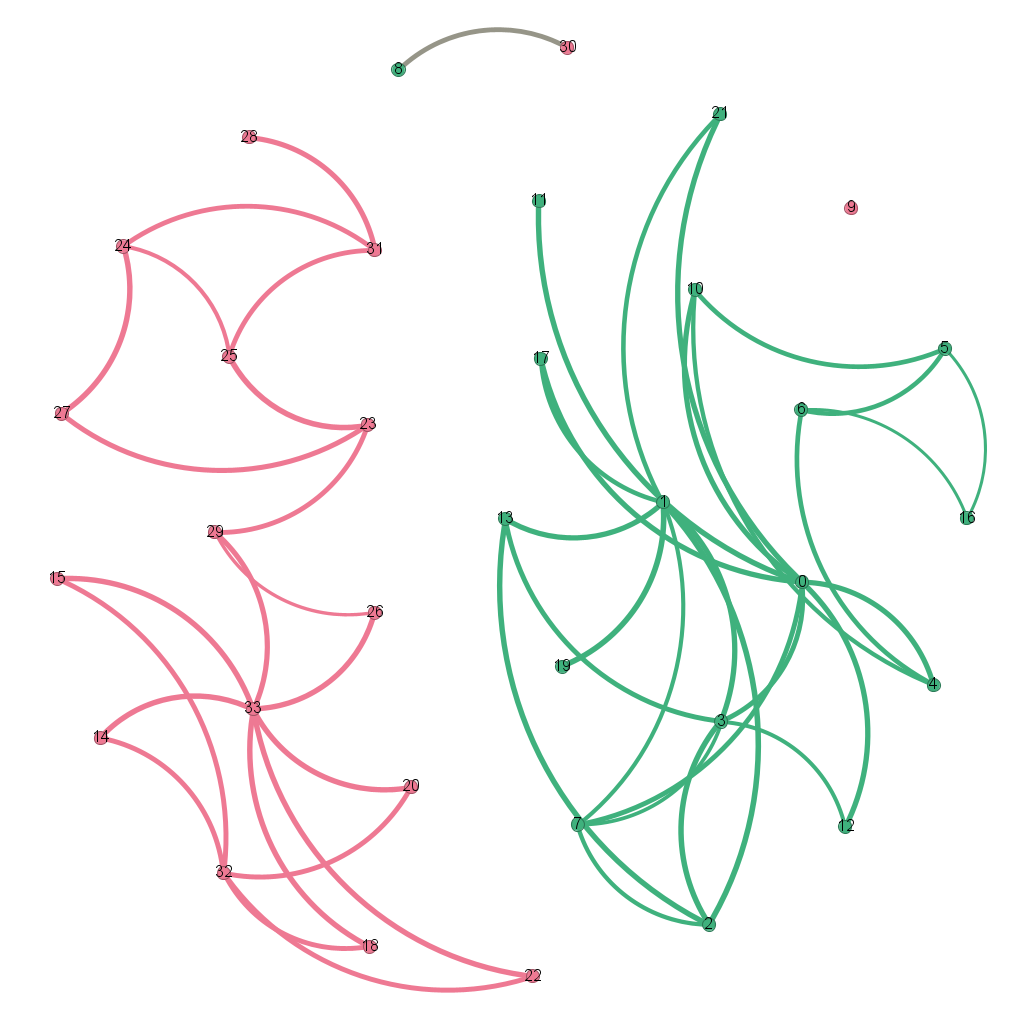}
        \caption{Karate (Detected communities)}
    \end{subfigure}
    \hfill
    \begin{subfigure}{0.28\textwidth}
        \centering
        \includegraphics[width=\textwidth]{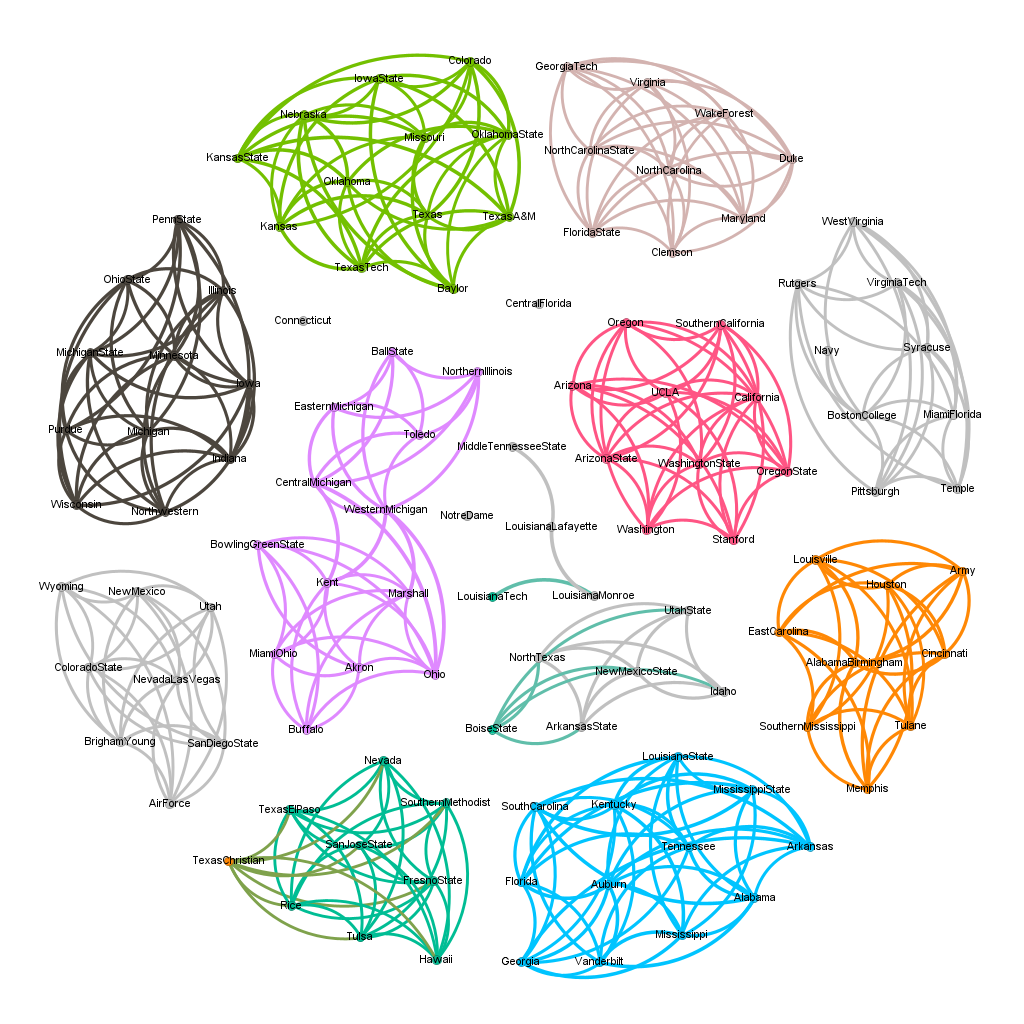}
        \caption{Football (Detected communities)}
    \end{subfigure}
    \hfill
    \begin{subfigure}{0.28\textwidth}
        \centering
        \includegraphics[width=\textwidth]{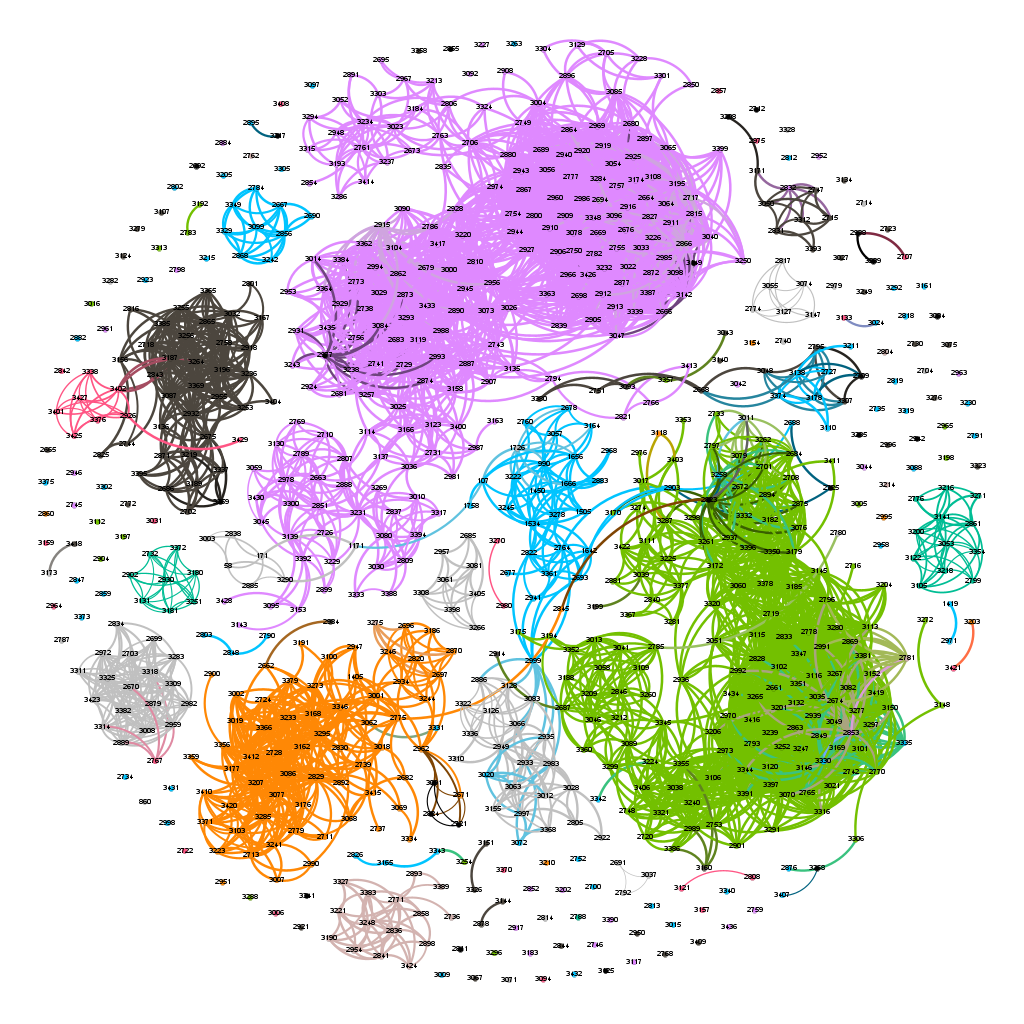}
        \caption{Facebook (Detected communities)}
    \end{subfigure}

    \caption{Visualization of community detection results}
    \label{fig:entropy_flow_visualization}
\end{figure}

\subsection{Node classification}

Node classification is another core task in graph learning, which aims to predict labels for unlabeled nodes based on both node features and the underlying graph structure. Graph neural networks (GNNs) have achieved remarkable success on this task by iteratively aggregating information from local neighborhoods \cite{WLHamilton, Kipf, Velickovic}. Recently, geometric graph learning methods, especially curvature-based approaches, have attracted increasing attention for their ability to characterize complex structural properties of graphs and improve representation learning \cite{Shi, Ye}. Among them, GEGCN \cite{GEGCN} uses discrete Ollivier Ricci flow to model graph evolution, and achieves excellent performance on both homophilic and heterophilic node classification tasks.

\subsubsection{Algorithm and experimental setup}

To evaluate the effectiveness of the proposed flows on node classification, we adopt the GEGCN framework \cite{GEGCN} as our backbone architecture, replacing its Ricci flow-based graph evolution with our entropy-driven mechanisms. Specifically, we consider two strategies derived from our proposed flows: LEF alone and CLEF (CF+LEF). For each strategy, we first evolve edge weights over multiple iterations to generate temporal sequences of weights and entropy. Following the GEGCN pipeline, these sequences are encoded by an LSTM to capture structural dynamics, producing edge importance scores that are then incorporated into a graph convolutional network for label prediction. Our primary aim is to investigate whether the temporal sequences generated by the proposed flows can effectively guide graph evolution and improve node representation learning, while retaining the full GEGCN architecture for fair comparison.

For clarity, we present the complete procedure for CLEF in Algorithm~\ref{alg:ef-gegcn}; LEF follows the same pipeline, with the only difference being the absence of the cohesion flow step.

\begin{algorithm}[htbp]
\caption{Node classification via CLEF (CF+LEF)}
\label{alg:ef-gegcn}
\KwIn{Graph $G=(V,E, \mathbf{w}_0)$ with node features $X$; parameters $\alpha \in (0,1)$; step sizes $s_c, s_e > 0$; number of cohesion iterations $N_c$; number of entropy iterations $N_e$.}
\KwOut{Predicted node labels $\hat{Y}$.}

\textbf{Step 1: Cohesion-driven graph evolution}

\For{$i = 0,1,\dots,N_c-1$}{
    Compute edge cohesion $C_e(t_i)$ for all $e \in E$\;
    Update edge weights $w_e(t_i+1) \leftarrow w_e(t_i) + s_c \cdot C_e(t_i)$\;
}

\textbf{Step 2: Entropy-driven graph evolution}

\For{$j = 0,1,\dots,N_e-1$}{
    Compute edge entropy $\Theta_e(t_j)$ for all $e \in E$\;
    Update edge weights $w_e(t_j+1) \leftarrow w_e(t_j) + s_e \cdot \Theta_e(t_j)$\;
}
Obtain edge weight sequences $\{w_e(t)\}_{t=0}^{N_c+N_e}$ and entropy sequences $\{\Theta_e(t)\}_{t=0}^{N_c+N_e}$\;

\textbf{Step 3: Structural dynamics encoding}

Encode the joint sequence $\{[\Theta_e(t), w_e(t)]\}_{t=0}^{N_c+N_e}$ using an LSTM to obtain edge representations\;
Compute edge importance scores $a_e^*$ from the edge representations\;
Construct the weighted adjacency matrix $\tilde{A}^*$ and normalize it to $\hat{A}^*$\;

\textbf{Step 4: Entropy-aware graph convolution}

Initialize node representations $H^{(0)} = X$\;
\For{$\ell = 0,1,\dots,L-1$}{
    $H^{(\ell+1)} = \sigma(\hat{A}^* H^{(\ell)} W^{(\ell)})$\;
}
Compute predicted labels $\hat{Y} = \mathrm{softmax}(H^{(L)})$\;

\Return $\hat{Y}$.
\end{algorithm}

We evaluate the proposed framework on 10 benchmark datasets covering both homophilic and heterophilic graph structures.
\begin{itemize}
    \item Homophilic datasets: Cora, Citeseer, Pubmed \cite{Sen}, Coauthor CS, Amazon Photos \cite{McAuley}. Following the experimental protocol in \cite{Shi}, we partition the nodes into 20\% for training, 10\% for validation, and 70\% for test.
    \item Heterophilic datasets: Cornell, Texas, Wisconsin (WebKB collection) \cite{Craven}, Chameleon \cite{Rozemberczki}, Actor \cite{Tang}. We follow the split 60\%/20\%/20\% for training, validation, and testing, respectively \cite{Topping}. And the test set is always fixed rather than randomly sampled, and the remaining nodes are then randomly divided into training and validation sets. This is to ensure that the test set is not used for any training or hyperparameter tuning prior to final evaluation.
\end{itemize}
Detailed statistics of the datasets are summarized in Tables \ref{tab:homophilic} and \ref{tab:heterophilic}.

\begin{table}[H]
\centering
\caption{Statistics of homophilic datasets}
\label{tab:homophilic}
\small
\begin{tabular}{|l|c|c|c|c|c|}
\hline
\textbf{Dataset} & \textbf{Nodes} & \textbf{Edges} & \textbf{Features} & \textbf{Classes} & \textbf{Homophily (H(G))} \\
\hline
Cora & 2708 & 5278 & 1433 & 7 & 0.83 \\
Citeseer & 3327 & 4552 & 3703 & 6 & 0.72 \\
Pubmed & 19717 & 44324 & 500 & 3 & 0.79 \\
Coauthor CS & 18333 & 81894 & 6805 & 15 & 0.83 \\
Amazon Photos & 7487 & 119043 & 745 & 8 & 0.85 \\
\hline
\end{tabular}
\end{table}

\begin{table}[thbp]
\centering
\caption{Statistics of heterophilic datasets}
\label{tab:heterophilic}
\small
\begin{tabular}{|l|c|c|c|c|c|}
\hline
\textbf{Dataset} & \textbf{Nodes} & \textbf{Edges} & \textbf{Features} & \textbf{Classes} & \textbf{Homophily (H(G))} \\
\hline
Cornell & 183 & 280 & 1703 & 5 & 0.30 \\
Texas & 183 & 295 & 1703 & 5 & 0.11 \\
Wisconsin & 251 & 466 & 1703 & 5 & 0.21 \\
Chameleon & 2277 & 31421 & 2325 & 5 & 0.23 \\
Actor & 7600 & 26752 & 932 & 5 & 0.24 \\
\hline
\end{tabular}
\end{table}

We compare LEF and CLEF with a comprehensive collection of baseline methods. For homophilic datasets, the compared methods cover classical graph neural networks including MLP, GCN \cite{Kipf}, GAT \cite{Velickovic}, and GraphSAGE \cite{WLHamilton}, as well as models with enhanced information propagation capabilities, such as JKNet \cite{Xu}, APPNP \cite{Gasteiger}, GPRGNN \cite{Chien}, MoNet \cite{Monti}, and UFGConv \cite{Zheng}. We also incorporate geometry-based approaches, namely CurvGN \cite{Ye}, RC-UFG \cite{Shi}, and GEGCN.
For heterophilic benchmarks, we further adopt methods specifically designed to handle non-homophilous graph structures, including diffusion-based models such as DIGL \cite{Gasteiger2019b} and fully-adjacent augmentation (+FA) \cite{Alon}, as well as graph rewiring techniques exemplified by SDRF \cite{Topping}. Moreover, the global entropy flow (GEF) is included for evaluation under both homophilic and heterophilic settings, enabling a direct and fair performance comparison between the proposed local formulation and its global counterpart.

For LEF, we fix the step size $s=0.1$,  $\alpha=0.5$, and $N_e=10$ on all datasets.
For CLEF, we first apply CF with step size $s_c=0.1$ for $N_c=10$ iterations, then apply LEF with $s_e=0.01$ and $\alpha=0.5$ for $N_e=10$ iterations. For completeness, we also evaluate GEF with the same settings as LEF, and CGEF with the same two-stage settings as CLEF.
All results are averaged over 10 independent runs. Hyperparameters are selected via random search based on validation performance.

\subsubsection{Experimental results and analysis}

We first compare LEF and CLEF with the baseline methods. Table~\ref{tab:node_classification} reports the node classification accuracy on homophilic datasets.
On Cora, LEF achieves 86.6\%, slightly below GEGCN (86.7\%), and CLEF achieves 85.8\%. On Citeseer, LEF and CLEF achieve 74.4\% and 74.0\%, respectively, lower than GEGCN (76.6\%) and APPNP/GPRGNN (75.9\%). On Pubmed, both LEF (88.3\%) and CLEF (88.5\%) outperform GEGCN (87.4\%). On Coauthor CS, LEF and CLEF both achieve the highest accuracy (95.2\%). On Amazon Photos, LEF (94.5\%) and CLEF (94.6\%) outperform GEGCN (94.1\%).

Overall, both LEF and CLEF perform competitively against the baseline methods, with CLEF achieving slightly higher accuracy than LEF on Pubmed and Amazon Photos, while they tie on Coauthor CS.

\begin{table}[H]
\centering
\caption{Node classification accuracy (\%) on homophilic datasets (mean $\pm$ std over 10 runs)}
\label{tab:node_classification}
\small
\begin{tabular}{lccccc}
\hline
Method & Cora & Citeseer & Pubmed & Coauthor CS & Amazon Photos \\
\hline
MLP & 55.1$\pm$1.4 & 59.1$\pm$1.2 & 71.4$\pm$0.8 & 88.3$\pm$0.7 & 69.6$\pm$3.8 \\
MoNet & 81.7$\pm$0.4 & 71.2$\pm$0.7 & 78.6$\pm$0.5 & 90.8$\pm$0.6 & 91.2$\pm$1.3 \\
GCN & 81.5$\pm$0.5 & 70.9$\pm$0.5 & 79.0$\pm$0.3 & 91.1$\pm$0.5 & 91.2$\pm$1.2 \\
GraphSAGE & 79.2$\pm$7.7 & 71.6$\pm$1.9 & 77.4$\pm$2.2 & 91.3$\pm$2.8 & 91.4$\pm$1.4 \\
GAT & 83.0$\pm$0.7 & 72.5$\pm$0.7 & 79.0$\pm$0.3 & 90.5$\pm$0.6 & 85.1$\pm$2.3 \\
JKNet & 83.7$\pm$0.7 & 72.5$\pm$0.4 & 82.6$\pm$0.5 & 91.1$\pm$0.3 & 86.1$\pm$1.1 \\
APPNP & 83.5$\pm$0.7 & \second{75.9$\pm$0.6} & 80.2$\pm$0.3 & 91.5$\pm$0.1 & 87.0$\pm$0.9 \\
GPRGNN & 83.8$\pm$0.9 & \second{75.9$\pm$0.7} & 82.3$\pm$0.2 & 91.8$\pm$0.1 & 87.0$\pm$0.9 \\
CurvGN & 82.6$\pm$0.6 & 71.5$\pm$0.8 & 78.8$\pm$0.6 & 92.9$\pm$0.4 & 92.5$\pm$0.5 \\
UFGConv$\_$S & 83.0$\pm$0.5 & 71.0$\pm$0.6 & 79.4$\pm$0.4 & 92.1$\pm$0.2 & 92.1$\pm$0.5 \\
UFGConv$\_$R & 83.6$\pm$0.6 & 72.7$\pm$0.6 & 79.6$\pm$0.4 & 93.0$\pm$0.7 & 92.5$\pm$0.2 \\
RC-UFG (Hom) & 84.4$\pm$0.7 & 72.5$\pm$0.7 & 82.9$\pm$0.2 & 94.2$\pm$0.9 & 93.5$\pm$0.7 \\
RC-UFG (Het) & 80.6$\pm$0.4 & 71.7$\pm$0.6 & 79.6$\pm$0.4 & 90.4$\pm$1.2 & 89.5$\pm$1.9 \\
GEGCN & \first{86.7$\pm$1.2} & \first{76.6$\pm$1.1} & 87.4$\pm$0.4 & 93.2$\pm$0.4 & 94.1$\pm$0.5 \\
\midrule
LEF & \second{86.6$\pm$0.2} & 74.4$\pm$0.3 & \second{88.3$\pm$0.2} & \first{95.2$\pm$0.1} & \second{94.5$\pm$0.2} \\
CLEF & 85.8$\pm$0.3 & 74.0$\pm$0.3 & \first{88.5$\pm$0.1} & \first{95.2$\pm$0.0} & \first{94.6$\pm$0.2} \\
\hline
\end{tabular}
\end{table}

Table~\ref{tab:table2} reports the results on heterophilic datasets. On Cornell, LEF achieves 75.14\% and CLEF achieves 75.68\%, both substantially higher than GEGCN (68.61\%). On Texas, LEF achieves 78.38\% and CLEF achieves 79.46\%, compared with GEGCN (70.27\%). On Wisconsin, LEF achieves 81.96\% and CLEF achieves 83.14\%, compared with GEGCN (67.39\%). On Chameleon, LEF achieves 66.73\%, higher than GEGCN (60.72\%), while CLEF achieves 64.93\%. On Actor, LEF (35.31\%) and CLEF (35.51\%) are slightly lower than GEGCN (37.18\%). Overall, both LEF and CLEF outperform GEGCN on four out of five datasets, with CLEF achieving higher accuracy than LEF on Cornell, Texas, and Wisconsin, while LEF performs better on Chameleon.

\begin{table}[h]
\centering
\caption{Node classification accuracy (\%) on heterophilic datasets (mean $\pm$ std over 10 runs)}
\label{tab:table2}
\small
\begin{tabular}{lccccc}
\hline
Method & Cornell & Texas & Wisconsin & Chameleon & Actor \\
\hline
GCN & 52.69$\pm$0.21 & 61.19$\pm$0.49 & 54.60$\pm$0.86 & 41.33$\pm$0.18 & 23.84$\pm$0.43 \\
Undirected & 53.20$\pm$0.53 & 63.38$\pm$0.87 & 51.37$\pm$1.15 & 42.02$\pm$0.30 & 21.45$\pm$0.47 \\
+FA & 58.29$\pm$0.49 & 64.82$\pm$0.29 & 55.48$\pm$0.62 & 42.67$\pm$0.17 & 24.14$\pm$0.43 \\
DIGL & 58.26$\pm$0.50 & 62.03$\pm$0.43 & 49.53$\pm$0.27 & 42.02$\pm$0.13 & 24.77$\pm$0.32 \\
DIGL+Undirected & 59.54$\pm$0.64 & 63.54$\pm$0.38 & 52.23$\pm$0.54 & 42.68$\pm$0.12 & 25.45$\pm$0.30 \\
SDRF & 54.60$\pm$0.39 & 64.46$\pm$0.38 & 55.51$\pm$0.27 & 42.73$\pm$0.15 & 28.42$\pm$0.75 \\
SDRF+Undirected & 57.54$\pm$0.34 & 67.02$\pm$0.40 & 56.55$\pm$0.86 & 44.46$\pm$0.17 & 28.35$\pm$0.06 \\
GEGCN & 68.61$\pm$0.26 & 70.27$\pm$0.69 & 67.39$\pm$0.93 & 60.72$\pm$0.16 & \first{37.18$\pm$0.30} \\
\midrule
LEF & \second{75.14$\pm$2.13} & \second{78.38}$\pm$9.70 & \second{81.96}$\pm$7.67 & \first{66.73$\pm$0.82} & 35.31$\pm$0.52 \\
CLEF & \first{75.68$\pm$2.85} & \first{79.46$\pm$3.86} & \first{83.14$\pm$3.36} & \second{64.93}$\pm$4.08 & \second{35.51}$\pm$0.01 \\
\hline
\end{tabular}
\end{table}

As a supplementary comparison, we also examine the performance of GEF and CGEF alongside LEF and CLEF. Tables~\ref{tab:node_classification_four} and~\ref{tab:node_classification_four_het} report the results on homophilic and heterophilic datasets, respectively. The four strategies achieve comparable results across all datasets, with differences generally within 1\%, suggesting that the choice of entropy formulation and the inclusion of cohesion preprocessing have limited influence on node classification performance.

\begin{table}[thbp]
\centering
\caption{Comparison of four strategies on homophilic datasets (\%)}
\label{tab:node_classification_four}
\small
\begin{tabular}{lccccc}
\hline
Method & Cora & Citeseer & Pubmed & Coauthor CS & Amazon Photos \\
\hline
LEF & 86.6$\pm$0.2 & 74.4$\pm$0.3 & 88.3$\pm$0.2 & 95.2$\pm$0.1 & 94.5$\pm$0.2 \\
GEF & 86.3$\pm$0.2 & 73.2$\pm$0.3 & 88.3$\pm$0.1 & 94.9$\pm$0.1 & 94.7$\pm$0.1 \\
CLEF & 85.8$\pm$0.3 & 74.0$\pm$0.3 & 88.5$\pm$0.1 & 95.2$\pm$0.0 & 94.6$\pm$0.2 \\
CGEF & 86.4$\pm$0.3 & 74.9$\pm$0.2 & 88.6$\pm$0.1 & 94.9$\pm$0.1 & 94.8$\pm$0.1 \\
\hline
\end{tabular}
\end{table}

\begin{table}[thbp]
\centering
\caption{Comparison of four strategies on heterophilic datasets (\%)}
\label{tab:node_classification_four_het}
\small
\begin{tabular}{lccccc}
\hline
Method & Cornell & Texas & Wisconsin & Chameleon & Actor \\
\hline
LEF & 75.14$\pm$2.13 & 78.38$\pm$9.70 & 81.96$\pm$7.67 & 66.73$\pm$0.82 & 35.31$\pm$0.52 \\
GEF & 74.32$\pm$1.91 & 77.84$\pm$4.19 & 81.76$\pm$7.04 & 64.93$\pm$0.73 & 35.49$\pm$0.52 \\
CLEF & 75.68$\pm$2.85 & 79.46$\pm$3.86 & 83.14$\pm$3.36 & 64.93$\pm$4.08 & 35.51$\pm$0.01 \\
CGEF & 74.86$\pm$4.24 & 79.73$\pm$1.91 & 82.35$\pm$4.43 & 66.67$\pm$0.87 & 34.86$\pm$0.81 \\
\hline
\end{tabular}
\end{table}

\subsubsection{Ablation studies and sensitivity analysis}

We first conduct ablation studies and sensitivity analysis specifically for the LEF component, since the effectiveness of the cohesion flow has already been validated by comparing LEF with CLEF and GEF with CGEF in the main experiments. Here we focus on the contribution of the temporal entropy evolution modeling within the LEF framework.

We compare the full LEF model with four variants: (1) LAST: using only the entropy at the final iteration; (2) Mean: averaging the entropy over all iterations; (3) MLP: using a multi-layer perceptron to encode the entropy sequence; (4) Random: replacing the original entropy sequence with a random one to eliminate all temporal evolutionary patterns. All variants share the same graph convolution backbone.

\begin{figure}[H]
    \centering
    \includegraphics[width=0.63\linewidth]{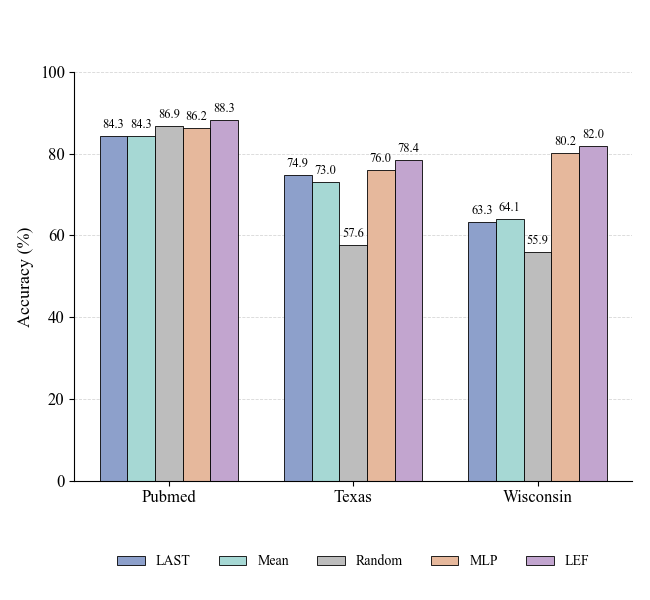}
    \caption{Ablation study results on test accuracy for LEF}
    \label{fig:ablation}
\end{figure}

As shown in Figure~\ref{fig:ablation}, the full LEF model consistently achieves the highest accuracy on all three datasets, confirming that modeling the temporal dynamics of entropy yields superior node representations compared to static or simple aggregation methods. The Random variant severely degrades performance on heterophilic datasets (Texas and Wisconsin), while remaining relatively high on the large homophilic dataset Pubmed. This indicates that meaningful temporal entropy information is critical for accurate classification on heterophilic graphs, while the original graph structure itself carries stronger signals on large homophilic graphs, making the model less sensitive to noise in edge importance.

Overall, the ablation results validate that the entropy-aware graph evolution within the LEF component effectively captures dynamic structural information, leading to robust performance improvements across diverse graph types.

We further analyze the sensitivity of LEF performance to the number of entropy flow iterations $N$, with results shown in Figure~\ref{fig:entropy_flow_duration}. For all three datasets, performance rises markedly with the increase of $N$, and then stabilizes or slightly drops when iterations become excessive. Across all datasets, the optimal range of $N$ lies between 6 and 8, achieving a good balance between classification accuracy and computational cost.

\begin{figure}[thbp]
\centering
\includegraphics[width=0.7\linewidth]{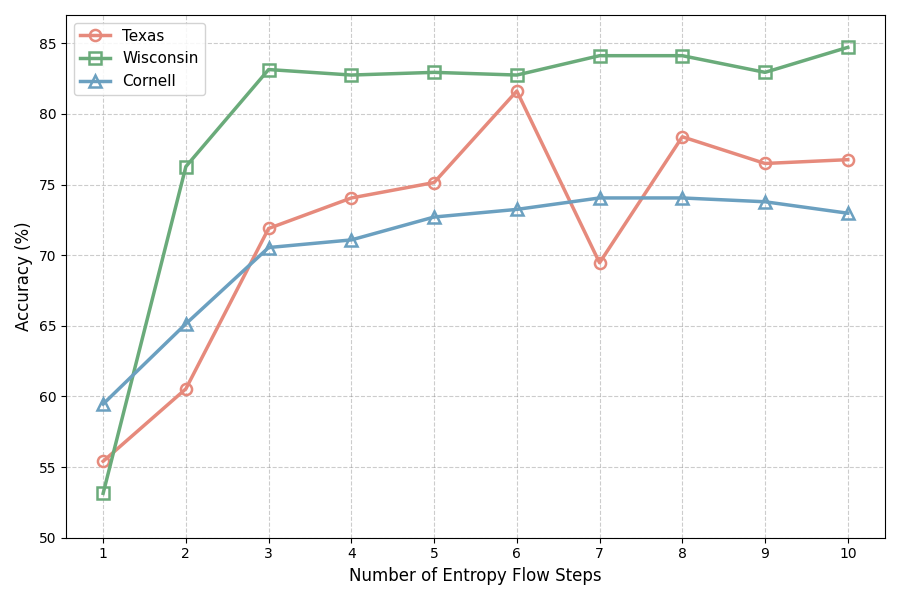}
\caption{Impact of the number of LEF iterations on test accuracy}
\label{fig:entropy_flow_duration}
\end{figure}

\subsection{Theoretical computational complexity analysis}
In this section, we provide a rigorous analysis of the computational complexity of LEF and CF, and compare them with GEF and discrete Ricci flow methods.

Let $G=(V,E,\mathbf{w})$ be a connected weighted graph with $n$ vertices, $m$ edges, and the average degree $D=2m/n$. Let $N$ be the number of iterations of the discrete entropy flow. We analyze the per-iteration complexity of LEF, CF, GEF  \cite{entropyflow}, and Ollivier Ricci flow  \cite{NLLG} as follows.

\paragraph{Cohesion flow (CF)}
For each edge $e=xy\in E$, computing the cohesion $C_{xy}$ requires evaluating the sets of exclusive neighbors $\mathscr{N}_x\setminus\mathscr{N}_y$, $\mathscr{N}_y\setminus\mathscr{N}_x$, and common neighbors $\mathscr{N}_{xy}$. These sets can be obtained by simple neighborhood lookups, and each involves at most $O(D)$ vertices. The quantities $A_{xy}$ and $B_{xy}$ are then computed via reciprocal weight summations over these sets, again with complexity $O(D)$. The normalization and exponential operation for $\widetilde C_{xy}$ and $C_{xy}$ are performed in constant time per edge. Consequently, computing the cohesion for all $m$ edges has a per-iteration complexity of $O(mD)$.

\paragraph{Local entropy flow (LEF)}
For each vertex $x$, computing the local random walk $\mu_x^\alpha$ only involves its 1-step neighborhood, with a complexity of $O(D)$. Thus, computing all local random walks for the entire graph has a complexity of $O(nD)=O(m)$. For each edge $e=xy$, computing the edge entropy $\Theta_e^\alpha$ involves summation over the set $\mathscr{N}_x\cup\mathscr{N}_y$, whose size is bounded by $O(D)$. Thus, computing the entropy for all $m$ edges has a complexity of $O(mD)$. The edge weight update step is element-wise, with a complexity of $O(m)$. Overall, the per-iteration complexity of LEF is $O(mD)$.

\paragraph{Global entropy flow (GEF)}
For GEF, the nowhere-zero random walk $\mu_x^\alpha$ has support on the entire vertex set $V$ (size $n$), so computing a single random walk has a complexity of $O(n)$. Computing all random walks for the graph has a complexity of $O(n^2)$. For each edge $e=xy$, computing the edge entropy involves summation over all $n$ vertices, leading to a per-iteration complexity of $O(mn)$ for all edges.

\paragraph{Ollivier Ricci flow}
The computational complexity of the discrete Ollivier Ricci flow is dominated by all-pairs shortest path calculation and optimal transport solving for each edge. We run Dijkstra's algorithm for every node to compute single-source shortest paths, with a per-node complexity of $O(m + n\log n)$ and a total cost of $O(nm + n^2\log n)$ across all $n$ nodes. Solving the discrete optimal transport problem between the two endpoints of an edge has a complexity of $O(D^3)$ \cite{NLLG, LaiX}, leading to a total cost of $O(mD^3)$ for all $m$ edges. Overall, the per-iteration complexity is $O(mn + n^2\log n + mD^3)$.

The sparse networks or scale-free networks are most common real-world networks. In the sparse scenario, we can assume that $D=O(1)$ (i.e. $n=O(m)$).  Under this assumption, both LEF and CF achieve a linear complexity $O(m)$, GEF reduces to a quadratic complexity $O(m^2)$, and the discrete Ricci flow suffers from a super-quadratic complexity $O(m^2\log m)$.
For scale-free networks, the average degree $D$ grows slowly with the number of vertices, typically as $D=O(\log n)$. Under this assumption, LEF and CF both maintain a nearly linear complexity $O(m\log n)$, with only a mild logarithmic factor overhead from the slowly growing average degree, GEF has a complexity $O(mn)$, and the discrete Ricci flow has a super-quadratic complexity $O(mn+ n^2\log n+m\log^3 n)$.

In summary, LEF and CF achieves a dramatic reduction in computational cost compared with both GEF and Ricci flow, especially for large-scale real-world networks, which is fully validated by our empirical running time experiments.

\section{Concluding remarks} \label{conclusion}

In this work, we propose two quantities for weighted graphs, local entropy and cohesion, and construct their associated flows, LEF and CF, for edge weight evolution. We further introduce a two-stage framework, CLEF, which applies CF followed by LEF for subsequent evolution.

We establish rigorous theoretical results for both flows, including global existence, uniqueness, and long-time convergence. Experimental results on community detection and node classification demonstrate the effectiveness of the proposed methods, with CLEF achieving competitive performance on benchmark networks. Complexity analysis confirms that LEF and CF attain near-linear complexity on sparse graphs, substantially outperforming GEF and discrete Ricci flow methods in scalability.

Overall, the proposed framework provides a theoretically grounded and computationally efficient approach to graph learning. Future work includes extensions to directed and dynamic graphs, as well as broader graph learning tasks.

\section*{Acknowledgements}
This research is partly supported by the National Natural Science Foundation of China (No. 12271039).

\section*{Declarations}

\noindent
\textbf{Data availability}:
All data needed are available freely at https://github.com/12tangze12/local-random-walk-based-entropy-flow.

\noindent
\textbf{Conflict of interest}: The authors declared no potential conflicts of interest with respect to the research, authorship, and publication of this article.

\noindent
\textbf{Ethics approval}: The research does not involve humans and/or animals. The authors declare that there are no ethics issues to be approved or disclosed.


\end{document}